\newcommand{\tle}{
Mean Equicontinuity and Related Properties in Hyperspace and Measure Dynamics
}

\documentclass[12pt, reqno]{amsart}
\usepackage{amsmath, amsthm, amscd, amsfonts, amssymb, graphicx, xcolor}
\usepackage{pifont}
\usepackage[bookmarksnumbered, colorlinks, plainpages]{hyperref}
\usepackage{todonotes}
\usepackage{multicol}
\usepackage{enumitem}
\usepackage{orcidlink}

\newtheorem{theorem}{Theorem}[section]
\newtheorem*{theorem*}{Theorem}
\newtheorem{lemma}[theorem]{Lemma}
\newtheorem{proposition}[theorem]{Proposition}
\newtheorem{corollary}[theorem]{Corollary}
\newtheorem*{corollary*}{Corollary}
\theoremstyle{definition}
\newtheorem{definition}[theorem]{Definition}

\newtheorem*{example*}{Example}

\theoremstyle{remark}
\newtheorem{remark}[theorem]{Remark}
\numberwithin{equation}{section}

\newcommand{\hyper}{\mathcal{H}}
\newcommand{\myper}{\mathcal{M}}

\newcommand{\wasser}{d}
\newcommand{\meanPM}{D}
\newcommand{\meanorbPM}{\mathcal{W}}
\newcommand{\diam}{\operatorname{diam}}
\newcommand{\Diam}{\operatorname{Diam}}
\newcommand{\ball}{\operatorname{B}}
\newcommand{\spa}{\operatorname{spa}}
\newcommand{\lip}[1]{\|#1\|_{\operatorname{Lip}}}
\newcommand{\dd}[1]{ \, \mathrm{d}#1}
\newcommand{\supp}{\operatorname{supp}}

\newcommand{\haar}[1]{\left|#1\right|}

\begin{document}
\setcounter{page}{1}

\color{black}{
\noindent 
}


\title[\tle]{\tle}

\author[]{Till Hauser \orcidlink{0000-0002-2580-3673}}
\address[T. Hauser]{Facultad de Matem\'aticas, Pontificia Universidad Cat\'olica de Chile. Edificio Rolando Chuaqui, Campus San Joaquín. Avda. Vicuña Mackenna 4860, Macul, Chile.}
\email{hauser.math@mail.de}
\author[]{Chunlin Liu \orcidlink{0000-0001-6277-013X}}
\address[C. Liu]{School of Mathematical Sciences, Dalian University of Technology, Dalian, 116024, P.R. China
and 
Institute of Mathematics, Polish Academy of Sciences, ul. Śniadeckich 8, 00-656 Warszawa, Poland
}
\email{chunlinliu@mail.ustc.edu.cn}
\thanks{This article was funded by the Deutsche Forschungsgemeinschaft (DFG, German Research Foundation) – 530703788. 
The second author was supported by the
Postdoctoral Fellowship Program and China Postdoctoral Science Foundation under Grant Number BX20250067, and the China Postdoctoral Science Foundation under Grant Number 2025M773074.
}
\begin{abstract}
    For a dynamical system $(X,T)$ we consider the induced dynamical systems $(\myper(X),T)$ and $(\hyper(X),T)$, consisting of Borel probability measures and closed non-empty subsets, respectively. We show that diam-mean equicontinuity of $(X,T)$ is equivalent to the diam-mean equicontinuity of $(\myper(X),T)$. 
    Furthermore, we establish that $(X,T)$ is mean equicontinuous, iff $(\myper(X),T)$ is mean equicontinuous, iff $(\myper(X),T)$ is weakly-mean equicontinuous. 
    For $(\hyper(X),T)$ the situation is different. 
    It is not hard to see that the diam-mean equicontinuity of $(\hyper(X),T)$ implies the diam-mean equicontinuity of $(X,T)$. 
    We provide examples for which $(X,T)$ is diam-mean equicontinuous, while $(\hyper(X),T)$ is not diam-mean equicontinuous.
    We prove that $(\hyper(X),T)$ is diam-mean equicontinuous, iff $(\hyper(X),T)$ is mean equicontinuous, iff $(\hyper(X),T)$ is weakly-mean equicontinuous. 
    We present our results in the context of continuous surjective maps $T\colon X\to X$ and discuss why they also hold for actions of locally compact $\sigma$-compact amenable groups. 
    
\noindent \textit{Keywords.}
Mean equicontinuity, 
Diam-mean equicontinuity, 
Weak-mean \linebreak equicontinuity, 
Hyperspace, 
Measure, 
Induced system,
Denjoy. 

\noindent \textit{2020 Mathematics Subject Classification.} 
Primary 
37B05;  
Secondary  
37B25, 
28A33,  
54B20. 
\end{abstract}

\maketitle

\section{Introduction}
Let $(X,T)$ be a dynamical system, i.e.\ $X$ be a compact metric space and $T\colon X\to X$ a continuous and surjective map. 
In classical statistical mechanics it is common to consider the dynamics imposed on the statistical states, i.e.\ the Borel probability measures $\mu$ on $X$. This point of view allows to 
study the deterministic dynamics on stochastic configurations and to incorporate incomplete knowledge of the starting point.
To apply the theory of topological dynamics it is natural to equip the set $\myper(X)$ of all Borel probability measures on $X$ with the compact and metrizable weak*-topology and to study the continuous push forward mapping $T_*\colon \myper(X)\to \myper(X)$. We simply denote $T:=T_*$ and write $(\myper(X),T)$ for this dynamical system. 
Naturally the question arises which dynamical properties of $(X,T)$ can be studied without a complete knowledge of the starting point, i.e.\ which are satisfied for $(X,T)$ if and only if they hold for $(\myper(X),T)$. 

A strongly related concept is that of the \emph{hyperspace} $\hyper(X)$, consisting of all closed and non-empty subsets of $X$. Equipped with the \emph{Hausdorff metric} (see Section \ref{sec:preliminaries}) $\hyper(X)$ is a compact metric space and the image mapping $A\mapsto T(A)$ is continuous. Again we simply denote $T\colon \hyper(X)\to \hyper(X)$ for the induced map and write $(\hyper(X),T)$ for the respective dynamical system. 
In analogy to $(\myper(X),T)$ is is also natural to ask which properties of $(X,T)$ are inherited by $(\hyper(X),T)$.

In the pioneering work of Bauer and Sigmund it was shown that $(X,T)$ is equicontinuous, 
iff $(\hyper(X),T)$ is equicontinuous, iff $(\mathcal{M}(X),T)$ is equicontinuous \cite[Prop.\ 7]{bauer1975topological}. 
Furthermore, Glasner showed that for minimal homeomorphisms~$T$ the induced system $(\mathcal{M}(X), T)$ is equicontinuous, iff it is distal, iff it is pointwise almost periodic \cite[Theorem 5.1]{Glasner1987}. 
More recently, the analog (for general dynamical systems) was proven by Akin, Auslander and Nagar, i.e.\ it was shown that $(\hyper(X),T)$ is equicontinuous, iff it is distal, iff it is pointwise almost periodic \cite[Theorem~6.9]{akin2017dynamics}. 
Note that this result was independently established by Li, Oprocha, Ye, and Zhang in \cite[Theorem~3.4]{LiOprochaYeZhang2017}.
For the definitions and background on standard notions, such as minimality, equicontinuity, distality and pointwise almost periodicity see \cite{auslander1988minimal, kerr2016ergodic}.

\begin{figure}[h]
\centering
\begin{tabular}{ccccc}
$\mathcal{M}(X)$ eq.   & $\Leftrightarrow$ & $X$ eq.      & $\Leftrightarrow$ & $\mathcal{H}(X)$ eq. \\
$\Updownarrow^*$       &                   & $\Downarrow$ &                   & $\Updownarrow$       \\
$\mathcal{M}(X)$ dist. & $\Rightarrow$     & $X$ dist.    & $\Leftarrow$      & $\mathcal{H}(X)$ dist. \\
$\Updownarrow^*$       &                   & $\Downarrow$ &                   & $\Updownarrow$       \\
$\mathcal{M}(X)$ pwap   & $\Rightarrow$     & $X$ pwap      & $\Leftarrow$      & $\mathcal{H}(X)$ pwap   
\end{tabular}
\caption{All implications between equicontinuity (eq.), distality (dist.) and pointwise almost periodicty (pwap). 
For $\Updownarrow^*$ the nontrivial direction $\Uparrow$ is only known for minimal homeomorphisms. }
\end{figure}

Furthermore, in \cite[Theorem 1 and Proposition 1]{bauer1975topological} it was shown that $(X,T)$ is weakly mixing, iff $(\mathcal{M}(X),T)$ is weakly mixing, iff $(\hyper(X),T)$ is weakly mixing. 
This result was extended by Banks, which proved that $(\hyper(X),T)$ is weakly mixing, iff  $(\hyper(X),T)$ is transitive \cite[Theorem 2]{banks2005chaos}. 
Also the analog holds for $(\myper(X),T)$, for which we provide the short proof in the appendix, where also the definitions of transitivity and weak mixing can be found. 

\begin{figure}[ht]
\centering
\begin{tabular}{ccccc}
$\mathcal{M}(X)$ wm     & $\Leftrightarrow$ & $X$ wm        & $\Leftrightarrow$ & $\mathcal{H}(X)$ wm \\
$\Updownarrow$          &                   & $\Downarrow$  &                   & $\Updownarrow$       \\
$\mathcal{M}(X)$ trans. & $\Rightarrow$     & $X$ trans.    & $\Leftarrow$      & $\mathcal{H}(X)$ trans.  
\end{tabular}
\caption{All Implications between weak mixing (wm) and transitivity (trans.).}
\end{figure}




It is not hard to construct a dynamical system $(X,T)$ with zero topological entropy for which $(\hyper(X),T)$ has positive topological entropy \cite[Page~666]{glasner1995quasi}. 
Note that (with significantly more afford in the construction) this phenomenon can also be observed for minimal dynamical systems $(X,T)$ as presented in \cite[Section~4]{glasner1995quasi}.
Note that the so far mentioned results show a remarkable parallelism between the induced dynamics on $\mathcal{M}(X)$ and $\hyper(X)$. 
It thus came as a surprise when Glasner and Weiss discovered that zero topological entropy of $(X,T)$ is inherited by $(\mathcal{M}(X),T)$ (for homeomorphisms $T$) \cite[Theorem A]{glasner1995quasi}.  
A similar phenomenon can be observed for the notions of nullness and tameness. For the definition and background see  \cite{kerr2007independence, LiLiu2025}. 
In \cite[Theorem~5.10]{kerr2005dynamical} it was shown (for homeomorphisms), that $(X,T)$ is null, iff $(\myper(X),T)$ is null. Furthermore, it follows from \cite[Theorem~5.3]{kohler1995enveloping} that $(X,T)$ is tame, iff $(\myper(X),T)$ is tame.  
To see that the analog does not hold for $(\hyper(X),T)$ note first that from the discussion in \cite{kerr2007independence} it follows that any null system is tame and that any tame system has zero topological entropy. 
It follows from \cite[Theorem~2.13]{glasner2018circularly} that any orientation preserving homeomorphism on the circle $(\mathbb{T},T)$ is null, and hence also tame.
Now consider a \emph{Denjoy} system $(\mathbb{T},T)$, i.e.\ an orientation preserving homeomorphism $T\colon \mathbb{T}\to \mathbb{T}$ that is not minimal and contains no periodic points. 
From the arguments presented in the second half of the proof of \cite[Theorem~5]{lampart2010topological} it follows that $(\hyper(\mathbb{T}),T)$ has infinite topological entropy and is thus neither tame, nor null. 
This shows that both tameness and nullness are not inherited by $(\hyper(X),T)$. 

\begin{figure}[ht]
\centering
\begin{tabular}{ccccc}
$\mathcal{M}(X)$ null     & $\Leftrightarrow$ & $X$ null        & $\Leftarrow^*$ & $\mathcal{H}(X)$ null \\
$\Downarrow$          &                   & $\Downarrow$  &                   & $\Downarrow^?$       \\
$\mathcal{M}(X)$ tame & $\Leftrightarrow$     & $X$ tame    & $\Leftarrow^*$      & $\mathcal{H}(X)$ tame  \\
$\Downarrow$          &                   & $\Downarrow$  &                   & $\Downarrow^?$       \\
$\mathcal{M}(X)$ has 0 entr.\     & $\Leftrightarrow$ & $X$ has 0 entr.\        & $\Leftarrow$    & $\mathcal{H}(X)$ has 0 entr..\
\end{tabular}
\caption{All implications between nullness, tameness and  zero topological entropy (has 0 entr.) for an invertible dynamical system $(X,T)$. 
For $\Downarrow^?$ it remains open, whether the converse holds.
For $\Leftarrow^*$ it seems to be open, whether the converse holds under the additional assumption of minimality of $(X,T)$. }
\end{figure}

In the endeavor of studying systems with discrete spectrum, Fomin  introduced in \cite{fomin1951dynamical} a property nowadays called \emph{mean equicontinuity}. Since then it has been studied intensively and numerous characterizations were found. See  \cite{li2015mean, downarowicz2016isomorphic, li2021meanEquicontinuity, fuhrmann2022structure} 
and the references within for further details on this notion. 
An action $(X,T)$ is called \emph{mean-equicontinuous} if for all $\epsilon>0$ there exists $\delta>0$ such that 
for all $x,x'\in X$ with $d(x,x')<\delta$ we have 
\[\limsup_{n\to \infty} \frac{1}{n}\sum_{k=0}^{n-1}d(T^kx,T^kx')<\epsilon.\]

A strongly related concept is the concept of weak-mean equicontinuity \cite{zheng2020new}.
An action is called \emph{weakly-mean equicontinuous} if for all $\epsilon>0$ there exists $\delta>0$ such that for all $x,x'\in X$ with $d(x,x')<\delta$ we have 
\[\limsup_{n\to \infty} \inf_{\sigma \in S_n}\frac{1}{n}\sum_{k=0}^{n-1}d(T^kx,T^{\sigma(k)}x')<\epsilon,\]
where $S_n$ is the group of all permutations of $\{0,\dots, n-1\}$. 
As presented\footnote{Note that the proof is presented in the setting of group actions but allows for a reformulation to the context of continuous surjections $(X,T)$.} in \cite[Theorem 1.10]{xu2024weak} $(X,T)$ is mean equicontinuous if and only if $X^2$ equipped with the mapping $(x,x')\mapsto(Tx,Tx')$ is weakly-mean equicontinuous. 

A third related concept is the concept of diam-mean equicontinuity. 
An action $(X,T)$ is called \emph{diam-mean equicontinuous} if for all $\epsilon>0$ there exists $\delta>0$ such that for all $x\in X$ we have 
\[\limsup_{n\to \infty} \frac{1}{n}\sum_{k=0}^{n-1}\diam(T^k(\ball_\delta(x)))<\epsilon,\]
where $\diam(A):=\sup_{x,x'\in A}d(x,x')$ denotes the \emph{diameter} of a subset $A\subseteq X$. 

Note that diam-mean equicontinuity implies mean-equicontinuity and that mean equicontinuity implies weak-mean equicontinuity. 
An example of a mean equicontinuous action that is not diam-mean equicontinuous can be found in \cite[Example 5.1]{downarowicz2015odometers} as illustrated in \cite[Page 3]{garcia2021mean}. 
It follows from \cite[Theorem 1.2]{zheng2020new} that any uniquely ergodic $(X,T)$ is weakly mean equicontinuous. 
Furthermore, note that it follows from the characterization of mean equicontinuity given in \cite[Theorem 1.1]{fuhrmann2022structure} and the variational principle that any (invertible) mean equicontinuous dynamical system $(X,T)$ has zero entropy.
Thus, any uniquely ergodic $(X,T)$ with positive entropy serves as an example of a weakly-mean equicontinuous action that is not mean-equicontinuous. 

It is natural to ask whether these notions are inherited by $\mathcal{M}(X)$ and $\hyper(X)$, and which of these properties are equivalent on $\hyper(X)$ and on $\mathcal{M}(X)$. 
A first insight is provided by \cite[Corollary 4.9]{garcia2019when}, where it is shown that $(\hyper(X),T)$ is diam-mean equicontinuous, iff it is mean equicontinuous. 
In Section \ref{sec:inducedDynamicsSubsets} we present an alternative proof for this statement. Our main result is the following complete picture of the behaviour of diam-mean equicontinuity, mean equicontinuity and weak-mean equicontinuity w.r.t.\ the induced dynamics. 

\begin{theorem}
\label{the:intro_Main}
    Let $(X,T)$ be a dynamical system. 
    \begin{enumerate}
    
        \item \label{intro_Main_M(X)_meanEq}
        The following statements are equivalent.
        \begin{itemize}
            \item[(i)] $(X,T)$ is mean equicontinuous.
            \item[(ii)] $(\myper(X),T)$ is mean equicontinuous. 
            \item[(iii)] $(\myper(X),T)$ is weakly-mean equicontinuous. 
        \end{itemize}
        
        \item \label{intro_Main_M(X)_diammeanEq}
        The following statements are equivalent.
        \begin{itemize}
            \item[(i)] $(X,T)$ is diam-mean equicontinuous.
            \item[(ii)] $(\myper(X),T)$ is diam-mean equicontinuous. 
        \end{itemize}
        
        \item \label{intro_Main_H(X)}
        The following statements are equivalent.
        \begin{itemize}
            \item[(i)] $(\hyper(X),T)$ is diam-mean equicontinuous.
            \item[(ii)] $(\hyper(X),T)$ is mean equicontinuous. 
            \item[(iii)] $(\hyper(X),T)$ is weakly-mean equicontinuous. 
        \end{itemize}
    \item 
    \label{intro_Main_H(X)_remark}
        Whenever $(\hyper(X),T)$ is diam-mean equicontinuous, then $(X,T)$ is diam-mean equicontinuous. 
    \end{enumerate}
\end{theorem}

We provide the proof of (\ref{intro_Main_M(X)_meanEq}-\ref{intro_Main_M(X)_diammeanEq}) and (\ref{intro_Main_H(X)}) in the Sections \ref{sec:inducedDynamicsMeasures} and \ref{sec:inducedDynamicsSubsets}, respectively. 
Note that $(X,T)$ can be identified with a subsystem of $(\hyper(X),T)$ via $x\mapsto\{x\}$, which allows to observe  (\ref{intro_Main_H(X)_remark}). 
We next discuss that the converse in (\ref{intro_Main_H(X)_remark}) does not hold.
For this recall from our discussion above that for a Denjoy system $(\mathbb{T},T)$ we have that $(\hyper(\mathbb{T}),T)$ has infinite topological entropy and hence that $(\hyper(\mathbb{T}),T)$ is not mean equicontinuous, i.e.\ not diam-mean equicontinuous. 
Nevertheless, we have the following. 

\begin{proposition}
    Any Denjoy system $(\mathbb{T},T)$ is diam-mean equicontinuous. 
\end{proposition}
\begin{proof}
    Recall that a continuous surjection $\pi\colon \mathbb{T}\to \mathbb{T}$ is called monotone if $\pi^{-1}(y)$ is connected, i.e.\ a closed arc for all $y\in \mathbb{T}$. 
    It is well known that any Denjoy system allows for a monotone factor map $\pi\colon \mathbb{T}\to \mathbb{T}$ onto an irrational rotation $(\mathbb{T},R)$. For details see \cite[Chapter 11]{katok1995introduction}. 
    It follows that $\pi$ allows for at most countable many $y\in \mathbb{T}$ with $\pi^{-1}(y)$ containing more than one point. 
    Denote $Y_0:=\{y\in \mathbb{T};\, |\pi^{-1}(y)|=1\}$. 
    For the unique invariant measure $\lambda$ (the Lebesgue measure) on $(\mathbb{T},R)$ we observe that $\lambda(Y_0)=1$. 
    This property of $\pi$, called \emph{regularity}, implies the diam-mean equicontinuity as shown in \cite[Remark~5.1 and Theorem~7.11]{hauser2025meandiameterregularitydiammean}.  
\end{proof}

\begin{remark}
    In forthcoming joint work with Maria Isabel Cortez, we show that Sturmian subshifts $(X,T)$ provide examples of non-equicontinuous systems for which $(\hyper(X),T)$ is diam-mean equicontinuous.
\end{remark}

\begin{figure}[ht]
\centering
\begin{tabular}{ccccc}
$\mathcal{M}(X)$ eq. & $\Leftrightarrow$ & $X$ eq.      & $\Leftrightarrow$ & $\mathcal{H}(X)$ eq. \\
$\Downarrow$         &                   & $\Downarrow$ &                   & $\Downarrow$       \\
$\mathcal{M}(X)$ dme & $\Leftrightarrow$ & $X$ dme      & $\Leftarrow$      & $\mathcal{H}(X)$ dme \\
$\Downarrow$         &                   & $\Downarrow$ &                   & $\Updownarrow$       \\
$\mathcal{M}(X)$ me  & $\Leftrightarrow$ & $X$ me       & $\Leftarrow$      & $\mathcal{H}(X)$ me  \\
$\Updownarrow$       &                   & $\Downarrow$ &                   & $\Updownarrow$       \\
$\mathcal{M}(X)$ wme & $\Rightarrow$     & $X$ wme      & $\Leftarrow$      & $\mathcal{H}(X)$ wme   
\end{tabular}
\caption{Our main results: All implications between equicontinuity (eq.), diam-mean equicontinuity (dme), mean equicontinuity (me) and weak-mean equicontinuity (wme). 
}
\end{figure}

A related concept is that of almost diam-mean equicontinuity \cite{garcia2021mean}, which we introduce in Section~\ref{sec:almostDME}. After observing the asymmmetric behaviour of diam-mean equicontinuity w.r.t.\ $\hyper(X)$ and $\myper(X)$ it needs to be highlighted that almost diam-mean equicontinuity behaves significantly different. 
In Section~\ref{sec:almostDME} we show that for an almost diam-mean equicontinuous dynamical system $(X,T)$ also $(\myper(X),T)$ and $(\hyper(X),T)$ are almost diam-mean equicontinuous. 

\begin{remark}
    Mean equicontinuity, weak-mean equicontinuity and (almost) diam-mean equicontinuity can also be studied in the context of actions of locally compact and $\sigma$-compact amenable groups. Our results also hold in this generality, which we will discuss in more detail in Section \ref{sec:groups}. 
\end{remark}

\subsection*{Organization of the text}
After discussing the relevant background in Section~\ref{sec:preliminaries} we show the results related to $\mathcal{M}(X)$ (Theorem~\ref{the:intro_Main}(\ref{intro_Main_M(X)_meanEq}-\ref{intro_Main_M(X)_diammeanEq})) in Section~\ref{sec:inducedDynamicsMeasures}. 
In Section \ref{sec:inducedDynamicsSubsets} we provide the proof of the results related to $\hyper(X)$ (Theorem \ref{the:intro_Main}(\ref{intro_Main_H(X)})). 
The concept of almost diam-mean equicontinuity is introduced and discussed in Section~\ref{sec:almostDME}.  
The discussion of the results in the context of actions of locally compact $\sigma$-compact amenable groups as well as the relevant background on such actions can be found in Section~\ref{sec:groups}. 

\subsection*{Convention}
We equip each topological space with the respective Borel $\sigma$-algebra and simply speak of measurability whenever we speak of measurability with respect to this $\sigma$-algebra. 
Given a compact  metric space $(X,d)$ we also write $d$ for the Hausdorff metric on $\hyper(X)$ and the Wasserstein metric on $\myper(X)$. For the justification of this convention see Remark \ref{rem:HausdorffConvention} and Remark \ref{rem:WassersteinConvention} below.

\section{Preliminaries}
\label{sec:preliminaries}
We write $A\Delta B:=(A\setminus B) \cup (B\setminus A)$ for the symmetric difference of $A$ and $B$. We denote $C(X)$ for the Banach space of all (real-valued) continuous functions on a compact metric space $X$. 
For $f\in C(X)$ we denote 
\begin{align*}
    \lip{f}:=\sup_{x,y\in X}\frac{|f(x)-f(y)|}{d(x,y)}.
\end{align*}
We denote $\mathbb{N}$ for the natural numbers without $0$ and $\mathbb{N}_0:=\mathbb{N}\cup \{0\}$.

\subsection{$\hyper(X)$}
Let $(X,d)$ be a compact metric space. 
We denote $\hyper(X)$ for the \emph{hyperspace of $X$}, i.e.\ the set of all closed and non-empty subsets of $X$. 
For $\epsilon>0$ and $x\in X$ we write
$\ball_\epsilon(x):=\{x'\in X;\, d(x,x')\leq \epsilon\}$ for the \emph{closed ball}. 
For $A\in \hyper(X)$ and $\epsilon>0$ we denote $\ball_\epsilon[A]:=\bigcup_{x\in A}\ball_\epsilon(x)$. 
The \emph{Hausdorff metric} is given by 
\begin{align*}
    d_\hyper(A,B):=\min\{\epsilon>0;\, A\subseteq \ball_\epsilon[B] \text{~~~and~~~}B\subseteq \ball_\epsilon[A]\},
\end{align*}
for $A,B\in \hyper(X)$.
For $x\in X$ and $A\in \hyper(X)$ we denote $d(x,A):=\min_{y\in A}d(x,y)$. 
Note that, as a straightforward consequence of compactness, the infimuma in the above formulas are attained. 

\begin{remark}\cite[Section 3.2]{beer1993topologies}
\label{rem:hausdorff_metric_formulas}
    For $A,B\in \hyper(X)$ we have 
    \begin{align*}
        d_\hyper(A,B)
        &= \max\left\{\max_{x\in A} d(x,B), \max_{x\in B}d(x,A)\right\}\\
        &=\max_{x\in X} |d(x,A)-d(x,B)|. 
    \end{align*}
\end{remark}
It is well known that $d_\hyper$ is a metric on $\hyper(X)$ and that $\hyper(X)$ equipped with the Hausdorff metric $d_\hyper$ is a compact metric space \cite{nadler1992continuum}. 
For $A\in \hyper(X)$ the \emph{diameter of $A$} is given by $\diam(A)=\max_{x,x'\in A}d(x,x')$. 

\begin{remark}
\label{rem:lipschitzContinuity}
    For $x\in X$ and $A\in \hyper(X)$ we have  
    $d(x,\cdot),\diam\in C(\hyper(X))$,
    $d(\cdot,A)\in C(X)$,
    $\lip{d(x,\cdot)}=1$,
    $\lip{\diam}\leq 2$,  and 
    $\lip{d(\cdot,A)}\leq 1$. 
\end{remark}

\begin{remark}
\label{rem:HausdorffConvention}
    Note that $X\to \hyper(X)$ given by $x\mapsto \{x\}$ is an isometric embedding and allows to identify $X$ as a closed subset of $\hyper(X)$. 
    Since $d_\hyper$ extends the metric of this subspace we will use the same symbol for the Hausdorff metric as for the metric on $X$ in the following discussion.    
\end{remark}

\subsection{$\myper(X)$}
Let $(X,d)$ be a compact metric space. 
We denote by $\myper(X)$ the set of all (Borel) probability measures on $X$. For a continuous map $\pi\colon X\to Y$ and $\mu\in \myper(X)$ we denote $\pi_*\mu$ for the \emph{push forward} of $\mu$ under $\pi$, which is defined by $\pi_*\mu(f):=\mu(f\circ \pi)$ for all $f\in C(Y)$. 

For $\mu_1,\mu_2\in \myper(X)$ a \emph{coupling} is a (Borel) probability measure $\kappa$ on $X^2$ with $\pi^{(i)}_*\kappa=\mu_i$, where $\pi^{(i)}\colon X^2\to X$ denote the respective projections for $i=1,2$.
We denote $\Pi(\mu_1,\mu_2)$ for the set of all couplings of $\mu_1$ and $\mu_2$. 
It follows from the continuity of $\pi^{(i)}_*$ that $\Pi(\mu_1,\mu_2)$ is a closed and hence compact subsets of $\myper(X^2)$. 
The \emph{Wasserstein metric} on $\myper(X)$ is defined by 
\[d_\myper(\mu_1,\mu_2):=\min_{\kappa\in \Pi(\mu_1,\mu_2)} \kappa(d).\]
\begin{remark}
\label{rem:preliminaries_WassersteinLipschitz}
    For $\mu_1,\mu_2\in \myper(X)$ we have 
    \begin{align*}
        d_\myper(\mu_1,\mu_2)
        =\sup_{f\in C(X);\, \|f\|_{\operatorname{Lip}}\leq 1} |\mu_1(f)-\mu_2(f)|.
    \end{align*}
\end{remark}
It is well known that the Wasserstein metric is a metric on $\myper(X)$ and that $\myper(X)$ equipped with $d_\myper$ is a compact metric space. The respective topology is the weak*-topology, given by the identification $\mathcal{M}(X)\subseteq C(X)^*$ via the Riesz-Markov-Kakutani representation theorem. 
For reference and further details see 
\cite[Chapter 7]{villani2003topics}. 

\begin{remark}
\label{rem:WassersteinConvention}
    For $x\in X$ we denote $\delta_x$ for the \emph{Dirac measure} at $x$.
    Note that $\delta\colon X\to \myper(X)$ given by $x\mapsto\delta_x$ is an isometric embedding and allows to identify $X$ as a closed subset of $\myper(X)$. 
    Since $d_\myper$ extends the metric of this subspace we will use the same symbol for the Wasserstein metric as for the metric on $X$ in the following discussion.    
\end{remark}

\subsection{Supports}
Let $X$ be a compact metric space. 
For $\mu\in \mathcal{M}(X)$ we denote $\supp(\mu)$ for the \emph{support of $\mu$}, i.e.\ the smallest closed subset $A\subseteq X$ that satisfies $\mu(A)=1$. It is given by the set of all $x\in X$ for which any open neighbourhood $U$ satisfies $\mu(U)>0$.

\subsection{Markov kernels and disintegration of measures}
Let $X$ and $Y$ be compact metric spaces.  
A \emph{Markov kernel} (also \emph{transition probability}) is a measurable function $\mu_{(\cdot)}\colon Y\to \mathcal{M}(X)$. 
Note that a function $\mu_{(\cdot)}$ is a Markov kernel if and only if for all $f\in C(X)$ the mapping $Y\ni y\mapsto \mu_y(f)$ is measurable, which in turn is equivalent to the measurability of $Y\ni y\mapsto \mu_y(A)$ for all measurable subsets $A\subseteq X$. 
For reference of the following see \cite[Theorem~10.7.2]{bogachev2007measure}.

\begin{lemma}
\label{lem:markovKernelmeasurabilityIntegrant}
    If $\mu_{(\cdot)}\colon Y\to \mathcal{M}(X)$ is a Markov kernel and $f\colon X\times Y\to \mathbb{R}$ is bounded and measurable, then 
    $Y\ni y\mapsto \int_X f(x,y)\dd\mu_y(x)$ is measurable.     
\end{lemma}

If $\pi\colon X\to Y$ is a continuous surjection and $\mu\in \mathcal{M}(X)$, then there exists a Markov kernel $\mu_{(\cdot)}\colon Y\to \mathcal{M}(X)$ such that $\supp(\mu_y)\subseteq \pi^{-1}(y)$ for all $y\in Y$ and such that $\mu=\int_Y\mu_y\dd (\pi_*\mu)(y)$ \cite[Proposition~10.4.12]{bogachev2007measure}. 
Such a Markov kernel $\mu_{(\cdot)}$ is called a \emph{disintegration of $\mu$ over $\pi$}.

If $\pi\colon X\to Z$ is a continuous surjection, then $\pi_*\colon \mathcal{M}(X)\to \mathcal{M}(Z)$ is continuous and hence measurable. 
For a Markov kernel $\mu_{(\cdot)}\colon Y\to \mathcal{M}(X)$ we denote $\pi_*\mu_{(\cdot)}$ for the Markov kernel $Y\mapsto \mathcal{M}(Z)$ defined by $y\mapsto \pi_*\mu_{y}$ and speak of the \emph{push forward of $\mu_{(\cdot)}$ under $\pi$}. 

Denote $\pi^{(1)}$ and $\pi^{(2)}$ for the respective projections $X^2\to X$.  
For $\mu,\nu\in \mathcal{M}(X)$ and $\kappa\in \Pi(\mu,\nu)$ we consider a disintegration $\kappa_{(\cdot)}$ of $\kappa$ over $\pi^{(1)}$. 
The Markov kernel $\nu_{(\cdot)}:=\pi^{(2)}_*\kappa_{(\cdot)}$ is called a \emph{disintegration of $\nu$ induced by $\kappa$}. 
Note that we have $\kappa_x=\delta_x\times \nu_x$ for all $x\in X$ and hence $\kappa=\int_X (\delta_x\times \nu_x) \dd\mu(x)$.
In particular, we have $\nu=\int_X \nu_x \dd \mu(x)$

\subsection{Dynamical systems}
A \emph{dynamical system} $(X,T)$ is given by a compact metric space $X$ and a continuous surjection $T\colon X\to X$. 
For $x\in X$ we denote $Tx:=T(x)$. 
A subset $A\in \hyper(X)$ is called invariant if 
$A\supseteq TA:=\{Tx;\, x\in X\}$. 
We also denote $T$ for the restriction of $T$ to a map $A\to A$ and speak of a \emph{subsystem $(A,T)$ of $(X,T)$.} 
A mapping $\pi\colon X\to Y$ between dynamical systems $(X,T)$ and $(Y,S)$ is called \emph{equivariant} if $\pi\circ T=S\circ \pi$.
An equivariant continuous surjection is called a \emph{factor map}.
 
Let $(X,T)$ be a dynamical system.
For $n\in \mathbb{N}$ and $(x_1,\dots, x_n)\in X^n$ we denote $T(x_1,\dots, x_n):=(Tx_1,\dots, Tx_n)$.
For $A\in \hyper(X)$ we denote 
$TA:=\{T.x;\, x\in X\}$. 
For $\mu\in \myper(X)$ we denote $T\mu:=T_*\mu$ for the push forward.
The respective mappings $X^n\to X^n$, $\hyper(X)\to \hyper(X)$ and $\myper(X)\to \myper(X)$ are continuous surjections. Since $(X,T)$ can be identified as a subsystem of $(X^n,T)$ (via $x\mapsto (x,\dots, x)$), of $(\hyper(X),T)$ (via $x\mapsto \{x\}$) and of $(\myper(X),T)$ (via $x\mapsto \delta_x$) we simply denote $T$ for the respective induced maps. 

For $f\in C(X)$ we denote $T^*f:=f\circ T$. 
For $n\in \mathbb{N}$ we denote 
\[F_n^*f:=\frac{1}{n}\sum_{k=0}^{n-1} (T^k)^*f
\hspace{1cm}\text{and}\hspace{1cm}
(F_n)_*\mu:=\frac{1}{n}\sum_{k=0}^{n-1} T^k\mu\]
for the \emph{Cesàro averages}. 
Note that we have  
$F_n^*f\in C(X)$, 
$(F_n)_*\mu\in \myper(X)$ and 
$(F_n)_*\mu(f)=\mu(F_n^*f)$.

\subsection{(Diam-/Weak-)mean equicontinuity}
\label{subsec:diamWeakMeanEquicontinuity}
For further reading on the following notions see 
\cite{zheng2020new, xu2024weak, li2021meanEquicontinuity, garcia2021mean, hauser2025meandiameterregularitydiammean}
and the references within. 
Let $(X,T)$ be a dynamical system. 
The \emph{$\mathcal{F}-$mean-orbital pseudometric} is defined by  
\begin{align*}
    \meanorbPM_\mathcal{F}(x,y):=\limsup_{n\to \infty}d((F_n)_*\delta_x,(F_n)_*\delta_y). 
\end{align*}
Note that $d\in C(X^2)$ and consider $(X^2,T)$. 
The \emph{$\mathcal{F}-$mean pseudo-metric} (also \emph{Besicowitch pseudometric}) is defined by 
\begin{align*}
    \meanPM_\mathcal{F}(x,y):=\limsup_{n\to \infty}(F_n^*d)(x,y). 
\end{align*}
Note that $\diam\in C(\hyper(X))$ and consider $(\hyper(X),T)$. 
For $A\in \hyper(X)$ the \emph{$\mathcal{F}$-mean diameter} is defined by 
\begin{align*}
    \Diam_\mathcal{F}(A):=\limsup_{n\to \infty}(F_n^*\diam)(A). 
\end{align*}
A point $x\in X$ is said to be a
\begin{itemize}
    \item 
    \emph{$\mathcal{F}$-weak-mean equicontinuity point} if for all $\epsilon>0$ there exists $\delta>0$ such that for all $y\in \ball_\delta(x)$ we have $\meanorbPM_\mathcal{F}(x,y)<\epsilon$. 
    \item 
    \emph{$\mathcal{F}$-mean equicontinuity point} if for all $\epsilon>0$ there exists $\delta>0$ such that for all $y\in \ball_\delta(x)$ we have $\meanPM_\mathcal{F}(x,y)<\epsilon$.
    \item 
    \emph{$\mathcal{F}$-diam-mean equicontinuity point} if for all $\epsilon>0$ there exists $\delta>0$ such that $\Diam_\mathcal{F}(\ball_\delta(x))<\epsilon$. 
\end{itemize}
$(X,T)$ is called \emph{weakly-mean equicontinuous} if all $x\in X$ are $\mathcal{F}$-weak-mean equicontinuity points. 
Similarly, we define the \emph{mean equicontinuity} and the \emph{$\mathcal{F}$-diam-mean equicontinuity} of $(X,T)$. 

\begin{remark}
    In Section \ref{sec:groups} we will consider these notions in the context of actions of $\sigma$-compact locally compact amenable groups. Adding the $\mathcal{F}$ to the notions allows to avoid a notational conflict below. 
\end{remark}

\begin{remark}
    These notions are independent from the choice of a continuous metric on $X$. For details on the argument see \cite[Lemma 3.4]{hauser2025meandiameterregularitydiammean}. 
\end{remark}

\begin{remark}
\label{rem:meanOrbitalReformulation}
\label{rem:meanPMDiam_relations}
\label{rem:prelims_basicRelationshipWME_ME_DME}
    As presented in \cite[Appendix A]{xu2024weak} it follows from the Birkhoff-von Neumann Theorem that 
    \[
    d((F_n)_*\delta_x,(F_n)_*\delta_y)=\inf_{\sigma\in S_n}\frac{1}{n}\sum_{k=0}^{n-1}d(T^kx,T^{\sigma(k)}y), 
    \]
    where $S_n$ denotes the set of all permutations on $\{0,\dots, n-1\}$. 
    In particular, for $x,y\in X$ we have 
    \[
        \meanorbPM_\mathcal{F}(x,y)
        \leq \meanPM_\mathcal{F}(x,y)
        =\Diam_\mathcal{F}(\{x,y\})
    .\]
    It follows that $(\mathcal{F}$-)diam-mean equicontinuity 
    implies $(\mathcal{F}$-)mean equicontinuity, and that
    $(\mathcal{F}$-)mean equicontinuity 
    implies ($\mathcal{F}$-)weak-mean equicontinuity.
\end{remark}

\begin{remark}
\label{rem:WMEProduct}
    $(X,T)$ is mean equicontinuous if and only if $(X^2,T)$ is weakly mean equicontinuous. This was shown in \cite{fuhrmann2022structure, xu2024weak} in the context of group actions and the respective arguments also apply for dynamical systems $(X,T)$. 
\end{remark}

The following lemma will be important in the following discussion and allows to observe (in combination with Remark \ref{rem:meanOrbitalReformulation}) that the notions defined here are equivalent to the notions defined in the introduction. 

\begin{lemma}
\label{lem:deltaIndependenceOfNotions}
    Let $A\in \hyper(X)$. 
    \begin{itemize}
        \item[(i)] 
        $A$ consists of $\mathcal{F}$-weak-mean equicontinuity points if and only if for all $\epsilon>0$ there exists $\delta>0$ such that for all $x\in A$ and $x'\in X$ with $d(x,x')<\delta$ we have $\meanorbPM_\mathcal{F}(x,x')<\epsilon$. 
        \item[(ii)] 
        $A$ consists of $\mathcal{F}$-mean equicontinuity points if and only if for all $\epsilon>0$ there exists $\delta>0$ such that for all $x\in A$ and $x'\in X$ with $d(x,x')<\delta$ we have $\meanPM_\mathcal{F}(x,x')<\epsilon$. 
        \item[(iii)]
        $A$ consists of $\mathcal{F}$-diam-mean equicontinuity points if and only if for all $\epsilon>0$ there exists $\delta>0$ such that for all $x\in A$ we have $\Diam_\mathcal{F}(\ball_\delta(x))<\epsilon$.        
    \end{itemize}
\end{lemma}
\begin{proof}
(i): 
    Clearly the right hand side implies the $\mathcal{F}$-weak-mean equicontinuity of all $x\in A$. 
    For the converse we argue by contraposition. 
    Assume that the latter is not satisfied. 
    Then there exists $\epsilon>0$ such that for all $n\in \mathbb{N}$ there exist $x_n\in A$ and $x_n'\in X$ with $d(x_n,x_n')<\tfrac{1}{n}$ and $\meanorbPM_\mathcal{F}(x_n,x_n')\geq \epsilon$. 
    Since $A$ is compact we assume w.l.o.g.\ that $x_n\to x$ for some $x\in A$.
    It follows that also $x_n'\to x$. 
    Since $A$ consists of $\mathcal{F}$-mean equicontinuity points there exists $\delta>0$ such that for all $y\in X$ with $d(x,y)<\delta$ we have $\meanorbPM_\mathcal{F}(x,y)<\epsilon/2$.
    For large $n\in \mathbb{N}$ we clearly have $d(x,x_n)<\delta$ and $d(x,x_n')<\delta$ and hence 
    \[
        \epsilon\leq \meanorbPM_\mathcal{F}(x_n,x_n')\leq \meanorbPM_\mathcal{F}(x_n,x)+\meanorbPM_\mathcal{F}(x,x_n')< \epsilon,
    \]
    a contradiction. 

(ii): 
    Follows from a similar argument as (i).

(iii): 
    Clearly the latter implies the former. For the converse we argue by contraposition and assume that the latter is not satisfied. 
    There exists $\epsilon>0$ and a sequence $(x_n)_{n\in \mathbb{N}}$ in $A$ such that $\Diam(\ball_{1/n}(x_n))\geq \epsilon$ for all $n\in \mathbb{N}$. 
    Since $A$ is compact we assume w.l.o.g.\ that $x_n\to x$ in $A$. 
    From $x\in A$ we know that $x$ is diam-mean equicontinuous and there exists $\delta>0$ such that $\Diam(\ball_\delta(x))<\epsilon$. 
    However, for sufficiently large $n\in \mathbb{N}$ we have $\ball_{1/n}(x_n)\subseteq \ball_\delta(x)$ and hence
    \[\epsilon\leq \Diam(\ball_{1/n}(x_n))\leq \Diam(\ball_\delta(x))<\epsilon,\]
    a contradiction. 
\end{proof}

\section{Induced dynamics on Borel probability measures}
\label{sec:inducedDynamicsMeasures}
Let $(X,T)$ be a dynamical system and denote $d$ for the metric on $X$. 
Recall that $x\mapsto \delta_x$ is an isometric embedding of $X$ into $\myper(X)$, which allows us to use the same symbol $d$ for the metric on $X$ and the Wasserstein metric on $\myper(X)$. 
Since this embedding is equivariant we can identify $(X,T)$ as a subsystem of $(\myper(X),T)$ and it follows that the $\mathcal{F}$-mean pseudometric on $(\myper(X),T)$ w.r.t.\ $d$ is an extension of the $\mathcal{F}$-mean pseudometric of $(X,T)$ w.r.t.\ $d$. 
This allows to write $D_\mathcal{F}$ for the $\mathcal{F}$-mean pseudometric of both, $(\myper(X),T)$ and $(X,T)$. 
Similarly, we simplify the notation and denote $\meanorbPM_\mathcal{F}$ for the $\mathcal{F}$-mean-orbital pseudometric, and $\Diam_\mathcal{F}$ for the $\mathcal{F}$-mean diameter of both, $(\myper(X),T)$ and $(X,T)$.

\subsection{Mean equicontinuity of $\myper(X)$}
We next show that any $\mu\in \myper(X)$ whose support consists of $\mathcal{F}$-mean equicontinuity points is a $\mathcal{F}$-mean equicontinuity point of $(\myper(X),T)$. For this we need the following lemmas. 

\begin{lemma}
\label{lem:MX_meanPseudometricPreparationsI}
    For $\mu,\nu\in \mathcal{M}(X)$ and 
    $\kappa\in \Pi(\mu,\nu)$ we have 
    \[(T^*d)(\mu,\nu)\leq \int_{X^2}(T^*d)(x,x')\dd \kappa(x,x')\]
\end{lemma}
\begin{proof}
    We have $T\kappa\in \Pi(T\mu,T\nu)$ and the definition of the Wasserstein metric yields 
    $
        (T^*d)(\mu,\nu)
        = d(T\mu,T\nu)
        \leq (T\kappa)(d)
        = \int_{X^2}(T^*d)(x,x')\dd \kappa(x,x').
    $
\end{proof}

\begin{lemma}
\label{lem:MX_meanPseudometricPreparationsII}
    For $\mu,\nu\in \mathcal{M}(X)$ and 
    $\kappa\in \Pi(\mu,\nu)$ we have 
    \[\meanPM_\mathcal{F}(\mu,\nu)\leq \int_{X^2}\meanPM_\mathcal{F}(x,x')\dd \kappa(x,x')\]
\end{lemma}
\begin{proof}
    For $k\in \mathbb{N}_0$ we observe from Lemma \ref{lem:MX_meanPseudometricPreparationsI} (applied to $(X,T^k)$) that 
    \[((T^k)^*d)(\mu,\nu)\leq \int_{X^2}((T^k)^*d)(x,x')\dd \kappa(x,x')\]
    and hence 
    \[(F_n^*d)(\mu,\nu)\leq \int_{X^2}(F_n^*d)(x,x')\dd \kappa(x,x').\]
    From Fatou's lemma we conclude
    \begin{align*}
        \meanPM_\mathcal{F}(\mu,\nu)
        &=\limsup_{n\to \infty}(F_n^*d)(\mu,\nu)\\
        &\leq \limsup_{n\to \infty}\int_{X^2}(F_n^*d)(x,x')\dd \kappa(x,x')\\
        &\leq \int_{X^2}\limsup_{n\to \infty}(F_n^*d)(x,x')\dd \kappa(x,x')\\
        &= \int_{X^2}\meanPM_\mathcal{F}(x,x')\dd \kappa(x,x').
        \qedhere
    \end{align*}
\end{proof}

\begin{proposition}
\label{pro:supportAndMeanEquicontinuousPoints}
    Let $(X,T)$ be a dynamical system and $\mu\in \mathcal{M}(X)$. 
    If $\supp(\mu)$ consists of $\mathcal{F}$-mean equicontinuity points, then $\mu$ is a $\mathcal{F}$-mean equicontinuity point of $(\mathcal{M}(X),G)$. 
\end{proposition}
\begin{proof}
    Let $\epsilon>0$. 
    By Lemma \ref{lem:deltaIndependenceOfNotions} there exists $\delta'>0$ such that for any $x\in \supp(\mu)$ and any $x'\in X$ with $d(x,x')<\delta'$ we have 
    $D_\mathcal{F}(x,x')<\epsilon/2$.    
    Denote $\delta:= \delta' \epsilon/(2\diam(X))$ and 
    $A:=\{(x,x')\in X^2;\, d(x,x')\geq\delta'\}$. 
    Consider $\nu\in \mathcal{M}(X)$ with $\wasser(\mu,\nu)<\delta$. 
    We next show that $D_\mathcal{F}(\mu,\nu)<\epsilon$. 
    
    By the definition of the Wasserstein metric there exists $\kappa\in \Pi(\mu,\nu)$ such that $\kappa(d)<\delta$. 
    We have
    \begin{align*}
        \int_A \delta' \dd\kappa
        \leq  \int_A d \dd\kappa
        \leq \kappa(d)
        < \delta = \frac{\delta'\epsilon}{2\diam(X)}.
    \end{align*}
    and hence 
    \begin{align}
        \label{enu:MXMeanEquicontinuityPtsI}
        \int_A D_\mathcal{F}\dd \kappa
        \leq\int_A \diam(X) \dd\kappa 
        < \frac{\epsilon}{2}. 
    \end{align}
    Furthermore, from $\kappa\in \Pi(\mu,\nu)$ we observe $\supp(\kappa)\subseteq \supp(\mu)\times X$. 
    Thus, for $(x,x')\in \supp(\kappa)\setminus A$ we have $x\in \supp(\mu)$ and $d(x,x')<\delta'$ and our choice of $\delta'$ yields $D_\mathcal{F}(x,x')<\epsilon/2$. 
    It follows that 
    \begin{align}
        \label{enu:MXMeanEquicontinuityPtsII}
        \int_{X^2\setminus A} D_\mathcal{F} d\kappa
        =\int_{\supp(\kappa)\setminus A} D_\mathcal{F} d\kappa
        \leq \frac{\epsilon}{2}. 
    \end{align}    
    Combining Lemma \ref{lem:MX_meanPseudometricPreparationsII} with \eqref{enu:MXMeanEquicontinuityPtsI} and \eqref{enu:MXMeanEquicontinuityPtsII} we conclude
    \begin{align*}
        D_\mathcal{F}(\mu,\nu)
        &\leq \int_{X^2} D_\mathcal{F}\dd \kappa
        < \epsilon.
        \qedhere
    \end{align*}
\end{proof}

\begin{proof}[Proof of Theorem \ref{the:intro_Main}(\ref{intro_Main_M(X)_meanEq}):]
    If $(X,T)$ is mean equicontinuous it follows from Proposition \ref{pro:supportAndMeanEquicontinuousPoints} that $(\myper(X),T)$ is mean equicontinuous. Since any mean equicontinuous action is weakly-mean equicontinuous it remains to show that the weak-mean equicontinuity of $(\myper(X),T)$ implies the mean equicontinuity of $(X,T)$. 
    
    For this consider the mapping $\phi\colon X^2\to \mathcal{M}(X)$ defined by $(x,x')\mapsto \tfrac{1}{3}\delta_x + \tfrac{2}{3}\delta_{x'}$ and note that $\phi$ is injective, equivariant and continuous. Since $X^2$ is compact $\phi$ is a conjugacy onto its image. 
    It follows that $(X^2,T)$ can be identified with the subsystem $(\phi(X^2),T)$ of $(\mathcal{M}(X),T)$. 
    Thus, whenever $(\mathcal{M}(X),T)$ is weakly mean equicontinuous, also $(X^2,T)$ is weakly mean equicontinuous and Remark~\ref{rem:WMEProduct}
    yields that $(X,T)$ is mean equicontinuous. 
\end{proof}

\subsection{Diam-mean equicontinuity of $\myper(X)$}
We next show the analog of Proposition~\ref{pro:supportAndMeanEquicontinuousPoints} for the notion of diam-mean equicontinuity. For this we need the following lemmas. 

\begin{lemma}
    \label{lem:meanDiameterPreparationsI}
    For $x\in X$, $\mu,\nu\in \mathcal{M}(X)$ and $\epsilon>0$ we have 
    \begin{align*}
        \epsilon (\mu \times \nu)(X^2\setminus \ball_\epsilon(x)^2)
        \leq (\delta_x\times \mu)(d) + (\delta_x\times \nu)(d). 
    \end{align*}
\end{lemma}
\begin{proof}
    Denote $B:=\ball_\epsilon(x)$. 
    We have 
    \begin{align*}
        \epsilon\mu(X\setminus B)
        =\int_{X\setminus B} \epsilon \dd \mu
        \leq \int_{X\setminus B}d(x,x')\dd \mu(x')
        \leq (\delta_x\times \mu)(d). 
    \end{align*}
    Similarly, we observe $\epsilon\nu(X\setminus B)\leq (\delta_x\times \nu)(d)$. 
    It follows that 
    \begin{align*}
        \epsilon(\mu\times \nu)(X^2\setminus B^2)
        &=\epsilon(\mu\times \nu)\left(((X\setminus B)\times X)\cup (X\times (X\setminus B))\right)\\
        &\leq \epsilon\mu(X\setminus B) +
        \epsilon\nu(X\setminus B) \\
        &\leq (\delta_x\times \mu)(d) + (\delta_x\times \nu)(d). 
        \qedhere
    \end{align*}
\end{proof}

Note that the map $X\ni x\mapsto (T^*\diam)(\ball_\epsilon(x))\in [0,\diam(X)]$ is continuous and hence integrable with respect to any $\mu\in \mathcal{M}(X)$. 

\begin{lemma}
    \label{lem:meanDiameterPreparationsII}
    Consider $\mu,\nu,\nu'\in \mathcal{M}(X)$. 
    For $\epsilon>0$ we have 
    \begin{align*}
        (T^*d)(\nu,\nu')
        \leq \int_X (T^*\diam)(\ball_\epsilon(x))\dd \mu(x)
        + \diam(X)\frac{d(\nu,\mu)+d(\mu,\nu')}{\epsilon}. 
    \end{align*}
\end{lemma}
\begin{proof}
    Consider $\kappa\in \Pi(\mu,\nu)$ with 
    $d(\mu,\nu)=\kappa(d)$. 
    Let $\nu_{(\cdot)}\colon X\to \mathcal{M}(X)$ be a disintegration of $\nu$ induced by $\kappa$. 
    Similarly, we choose $\kappa'\in \Pi(\mu,\nu')$ with $d(\mu,\nu')=\kappa(d)$ and a disintegration $\nu_{(\cdot)}'$ of $\nu'$ induced by $\kappa'$. 
    Consider the Markov kernel $\rho_{(\cdot)}\colon X\to \mathcal{M}(X^2)$ defined by $\rho_x:=\nu_x\times \nu_x'$ and denote $\rho:=\int_X \rho_x\dd\mu(x)$. 
    Note that $T\rho\in \Pi(T\nu,T\nu')$. 
    Thus, we have 
    \begin{align*}
        (T^*d)(\nu,\nu')
        &=d(T\nu,T\nu')
        \leq T\rho(d)
        =\rho(T^*d)
        =\int_X\int_{X^2}(T^*d)\dd \rho_x \dd\mu(x). 
    \end{align*}    
    Denote $A$ for the set of all $(x,x_1,x_2)\in X^3$ with $d(x,x_i)\leq \epsilon$ for $i=1,2$.
    Clearly, $A$ is closed and hence the characteristic function $\mathbf{1}_A\colon X^3\to \{0,1\}$ is measurable. 
    Furthermore, we have $A=\bigcup_{x\in X}\{x\}\times \ball_\epsilon(x)^2$. 
    The measurability of the integrants in the following computations thus follows from Lemma \ref{lem:markovKernelmeasurabilityIntegrant}.
    For $B\subseteq X$ we have 
    $T(B^2)=(TB)^2=:TB^2$ and compute
    \begin{align*}
        \int_X\int_{\ball_\epsilon(x)^2}(T^*d)\dd\rho_x\dd\mu(x)
        &=\int_X\int_{\ball_\epsilon(x)^2}d(T \overline{x})\dd\rho_x(\overline{x})\dd\mu(x)\\
        &=\int_X\int_{T\ball_\epsilon(x)^2}d(\overline{x})\dd\rho_x(\overline{x})\dd\mu(x)\\
        &\leq \int_X \diam(T\ball_\epsilon(x))\dd\mu(x)\\
        &=\int_X (T^*\diam)(\ball_\epsilon(x))\dd\mu(x).
    \end{align*}
    From $T^*d\leq \diam(X)$ and Lemma \ref{lem:meanDiameterPreparationsI} we observe
    \begin{align*}
        \int_X\int_{X^2\setminus \ball_\epsilon(x)^2}(T^*d)\dd\rho_x\dd\mu(x)
        &\leq \int_X\int_{X^2\setminus \ball_\epsilon(x)^2}\diam(X)\dd\rho_x\dd\mu(x)\\
        &= \diam(X) \int_X\rho_x(X^2\setminus \ball_\epsilon(x)^2)\dd\mu(x)\\
        &\leq \diam(X)\int_X \frac{1}{\epsilon}(\delta_x\times \nu_x(d) + \delta_x\times \nu_x'(d))\dd \mu(x)\\
        &= \diam(X) \frac{\kappa(d)+\kappa'(d)}{\epsilon}\\
        &= \diam(X) \frac{d(\nu,\mu)+d(\mu,\nu')}{\epsilon}. 
    \end{align*}    
    Combining our observations we conclude the statement.         
\end{proof}

\begin{lemma}
\label{lem:meanDiameterPreparationsIII}
    For $\mu\in \mathcal{M}(X)$ and $\epsilon,\delta>0$ we have 
    \begin{align*}
        (T^*\diam)(\ball_\delta(\mu))
        \leq \int_X (T^*\diam)(\ball_\epsilon(x))\dd \mu(x) + \frac{2\delta}{\epsilon}\diam(X).
    \end{align*}
\end{lemma}
\begin{proof}
    Consider $\nu,\nu'\in \ball_\delta(\mu)$ with 
    $d(T\nu,T\nu')
        =\diam(T\ball_\delta(\mu))$. 
    We have 
     \begin{align*}
        (T^*\diam)(\ball_\delta(\mu))
        &=\diam(T\ball_\delta(\mu))
        =d(T\nu,T\nu')
        =(T^*d)(\nu,\nu'). 
    \end{align*}
    From 
        $d(\mu,\nu)+d(\mu,\nu')\leq 2\delta$
    and Lemma \ref{lem:meanDiameterPreparationsII} we thus observe 
    \begin{align*}
        (T^*\diam)(\ball_\delta(\mu))
        &\leq \int_X (T^*\diam)(\ball_\epsilon(x))\dd \mu(x) + \frac{2\delta}{\epsilon}\diam(X). 
        \qedhere
    \end{align*}
\end{proof}

\begin{lemma}
\label{lem:meanDiameterPreparationsIV}
    For $\mu\in \mathcal{M}(X)$ and $\epsilon,\delta>0$ we have 
    \begin{align*}
        \Diam_\mathcal{F}(\ball_\delta(\mu))
        \leq \int_X \Diam_\mathcal{F}(\ball_\epsilon(x))\dd \mu(x) + \frac{2\delta}{\epsilon}\diam(X).
    \end{align*}
\end{lemma}
\begin{proof}
    From Lemma \ref{lem:meanDiameterPreparationsIII} applied to the dynamical system $(X,T^k)$ we observe that 
    \begin{align*}
        ((T^k)^*\diam)(\ball_\delta(\mu))
        \leq \int_X ((T^k)^*\diam)(\ball_\epsilon(x))\dd \mu(x) + \frac{2\delta}{\epsilon}\diam(X)
    \end{align*}
    and hence     
    \begin{align*}
        (F_n^*\diam)(\ball_\delta(\mu))
        \leq \int_X (F_n^*\diam)(\ball_\epsilon(x))\dd \mu(x) + \frac{2\delta}{\epsilon}\diam(X)
    \end{align*}
    The statement follows from taking the limit superior and Fatou's lemma. 
\end{proof}

\begin{proposition}
\label{pro:supportAndDiamMeanEquicontinuousPoints}
    Let $\mu\in \mathcal{M}(X)$. 
    If $\supp(\mu)$ consists of $\mathcal{F}$-diam-mean equicontinuity points, then $\mu$ is a $\mathcal{F}$-diam-mean equicontinuity point of $(\mathcal{M}(X),T)$.
\end{proposition}
\begin{proof}
    Let $\epsilon>0$. 
    From Lemma \ref{lem:deltaIndependenceOfNotions} we observe that there exists $\delta'>0$ such that for all $x\in \supp(\mu)$ we have $\Diam_\mathcal{F}(\ball_{\delta'}(x))<\epsilon/2$. 
    Denote $\delta:=\epsilon \delta'/(4\diam(X))$. 
    It follows from Lemma \ref{lem:meanDiameterPreparationsIV} that
    \begin{align*}
        \Diam_\mathcal{F}(\ball_{\delta}(\mu))
        &\leq \int_X \Diam_\mathcal{F}(\ball_{\delta'}(x))\dd\mu(x)
        + 2\diam(X)\frac{\delta}{\delta'}\\
        &\leq \int_X \frac{\epsilon}{2}\dd\mu(x) + \frac{\epsilon}{2}
        \leq \epsilon.
    \end{align*}
    This shows that $\mu$ is a $\mathcal{F}$-diam-mean equicontinuous point of $(\mathcal{M}(X),G)$.
\end{proof}

\begin{proof}[Proof of Theorem \ref{the:intro_Main}(\ref{intro_Main_M(X)_diammeanEq}):]
    Whenever $(X,T)$ is diam-mean equicontinuous it follows from Proposition \ref{pro:supportAndDiamMeanEquicontinuousPoints} that $(\myper(X),T)$ is diam-mean equicontinuous. 
    The converse follows, since diam-mean equicontinuity is inherited by subsystems and since $(X,T)$ can be identified with a subsystem of $(\myper(X),T)$ via $x\mapsto\delta_x$. 
\end{proof}

\section{Induced dynamics on closed subsets}
\label{sec:inducedDynamicsSubsets}
As above let $(X,T)$ be a dynamical system and denote $d$ for the metric on $X$. 
Note that $x\mapsto \{x\}$ gives an isometric and equivariant embedding of $(X,T)$ into $(\hyper(X),T)$. 
As in the previous discussion we use this fact to simplify our notation. 
We denote 
$D_\mathcal{F}$ for the $\mathcal{F}$-mean pseudometric, 
$\meanorbPM_\mathcal{F}$ for the $\mathcal{F}$-mean-orbital pseudometric, and 
$\Diam_\mathcal{F}$ for the $\mathcal{F}$-mean diameter 
of both, $(\hyper(X),T)$ and $(X,T)$. 
It was shown in \cite[Corollary 4.9]{garcia2019when} that $(\hyper(X),T)$ is diam-mean equicontinuous, iff it is mean equicontinuous. We next extend this result by proving that $(\hyper(X),T)$ is mean equicontinuous, iff it is weakly-mean equicontinuous.

\subsection{Weak mean equicontinuity of $\hyper(X)$}
We first show that the weak-mean equicontinuity of $(\hyper(X),T)$ implies the diam-mean equicontinuity of $(X,T)$. 
For this we use the following notion. 

\begin{definition}
    A subset $M\subseteq X$ is called \emph{$\epsilon$-spanning} if for any $x\in X$ there exists $y\in M$ with $d(x,y)\leq \epsilon$, i.e.\ if $\ball_\epsilon[M]=X$. 
    Since $X$ is compact it allows for a finite $\epsilon$-spanning subset. 
    We denote $\spa(\epsilon)$ for the minimal cardinality of an $\epsilon$-spanning subset of $X$. 
\end{definition}

\begin{lemma}
\label{lem:HX_PseudometricControlPreparationsI}
    For $\epsilon>0$ there exists $\phi\in C(\hyper(X))$ with $\lip{\phi}\leq 1$ such that for all $A,B\in \hyper(X)$ with $A\subseteq B$ we have 
    \begin{align*}
        d(A,B)\leq \spa(\epsilon)(\delta_A(\phi)-\delta_B(\phi))+2\epsilon. 
    \end{align*}
\end{lemma}
\begin{proof}
    Consider an $\epsilon$-spanning subset $E\subseteq X$ of cardinality $\spa(\epsilon)$. 
    For $x\in X$ we consider $\phi_x\colon \hyper(X)\to [0,\diam(X)]$ defined by $A\mapsto d(x,A)$. 
    We denote $\phi:=1/\spa(\epsilon) \sum_{x\in X}\phi_x$. 
    Recall from Remark \ref{rem:lipschitzContinuity} that $\lip{\phi_x}= 1$ and hence that $\lip{\phi}\leq 1$. 

    Consider $A,B\in \hyper(X)$ with $A\subseteq B$ and note that $d(x,A)\geq d(x,B)$ for all $x\in X$. 
    From Remark \ref{rem:lipschitzContinuity} we know that $d(\cdot, A), d(\cdot, B)\colon X\to [0,\infty)$ satisfy 
    $\lip{d(\cdot, A)}\leq 1$ and $ \lip{d(\cdot, B)}\leq 1$.
    From Remark \ref{rem:hausdorff_metric_formulas} and since $E$ is $\epsilon$-spanning we thus observe  
    \begin{align*}
        d(A,B)
        &=\sup_{x\in X} d(x,A)-d(x,B)\\
        &\leq 2\epsilon +\sup_{x\in E} (d(x,A)-d(x,B))\\
        &\leq 2\epsilon+\sum_{x\in E} (d(x,A)-d(x,B))\\
        &=2\epsilon+\sum_{x\in E} (\phi_x(A)-\phi_x(B))\\
        &=2\epsilon +\spa(\epsilon) (\phi(A)-\phi(B))\\
        &=2\epsilon +\spa(\epsilon)(\delta_A(\phi)-\delta_B(\phi)). 
    \qedhere
    \end{align*}
\end{proof}

\begin{lemma}
\label{lem:hyperWeylVsMeanOrb}
    Let $(X,T)$ be a dynamical system. 
    For $A,B\in \hyper(X)$ with $A\subseteq B$ and $\epsilon>0$ we have 
    \begin{align*}
        D_\mathcal{F}(A,B)\leq \spa(\epsilon)\meanorbPM_\mathcal{F}(A,B)+2\epsilon.
    \end{align*}
\end{lemma}
\begin{proof}
    Consider $\phi$ as in Lemma \ref{lem:HX_PseudometricControlPreparationsI}. 
    For $k\in \mathbb{N}_0$ we have $T^kA\subseteq T^kB$ and hence
    \begin{align*}
        ((T^k)^*d)(A,B)
        =d(T^kA,T^kB)
        &\leq \spa(\epsilon)(\delta_{T^kA}(\phi)-\delta_{T^kB}(\phi))+2\epsilon\\
        &= \spa(\epsilon)(T^k\delta_{A}(\phi)-T^k\delta_{B}(\phi))+2\epsilon.
    \end{align*}    
    For $n\in \mathbb{N}$ it follows that     
    \begin{align*}
        (F_n^*d)(A,B)\leq \spa(\epsilon)((F_n)_*\delta_A(\phi)-(F_n)_*\delta_B(\phi))+2\epsilon. 
    \end{align*}
    From $\lip{\phi}\leq 1$ and Remark \ref{rem:hausdorff_metric_formulas} we thus observe 
    \begin{align*}
        (F_n^*d)(A,B)\leq \spa(\epsilon)d((F_n)_*\delta_A,(F_n)_*\delta_B)+2\epsilon. 
    \end{align*}
    Taking the limit superior along $n$ we conclude the statement. 
\end{proof}

\begin{lemma}
    \label{lem:meanDiameterControlAtPointI}
    For $x\in X$ and $A\in \hyper(X)$ we have 
    $(T^*\diam)(A)\leq 2(T^*d)(\{x\},A).$
\end{lemma}
\begin{proof}
    Note that $\diam(T\{x\})=0$. We thus observe from $\lip{\diam}\leq 2$ that 
    \begin{align*}
        (T^*\diam)(A)
        &=|\diam(TA)-\diam(T\{x\})|\\
        &\leq 2d(T\{x\},TA)
        =2(T^*d)(\{x\},A).
        \qedhere
    \end{align*}
\end{proof}

\begin{lemma}
    \label{lem:meanDiameterControlAtPointII}
    For $x\in X$ and $A\in \hyper(X)$ we have 
    $\Diam_\mathcal{F}(A)\leq 2D_\mathcal{F}(\{x\},A)$. 
\end{lemma}
\begin{proof}
    Applying Lemma \ref{lem:meanDiameterControlAtPointI} to $(X,T^k)$ for $k\in \mathbb{N}_0$ yields
    \begin{align*}
        ((T^k)^*\diam)(A)\leq 2((T^k)^*d)(\{x\},A)
    \end{align*}
    and hence for all $n\in \mathbb{N}$ we have 
    \begin{align*}
        (F_n^*\diam)(A)\leq 2(F_n^*d))(\{x\},A)
    \end{align*}
    Taking the limit superior along $n$ we conclude the statement. 
\end{proof}

\begin{proposition}
\label{pro:hyper_HXwmeImpliesXdme}
    Consider $x\in X$. 
    If $\{x\}$ is a $\mathcal{F}$-weak-mean equicontinuity point of $(\hyper(X),T)$, then $x$ is a $\mathcal{F}$-diam-mean equicontinuity point of $(X,T)$. 
\end{proposition}
\begin{proof}
    Let $\epsilon>0$. 
    There exists $\delta>0$ such that $\mathcal{W}_\mathcal{F}(\{x\},\ball_\delta(x))<\epsilon/(4\spa(\epsilon/8))$. 
    Since $\{x\}\subseteq \ball_\delta(x)$ we observe from the Lemmas \ref{lem:hyperWeylVsMeanOrb} and \ref{lem:meanDiameterControlAtPointII} that 
    \begin{align*}
        \Diam_\mathcal{F}(\ball_\delta(x))
        &\leq 
        2D_\mathcal{F}(\{x\},\ball_\delta(x))\\
        &\leq 2\spa\left(\tfrac{\epsilon}{8}\right)\mathcal{W}_\mathcal{F}(\{x\},\ball_\delta(x))+4\frac{\epsilon}{8}\\
        &< \frac{\epsilon}{2}+\frac{\epsilon}{2}=\epsilon. 
        \qedhere
    \end{align*}    
\end{proof}

\subsection{Diam-mean equicontinuity points in $\hyper(X)$}
Under the assumption that $A\in \hyper(X)$ consists of $\mathcal{F}$-diam-mean equicontinuity points we next show that $A$ is a $\mathcal{F}$-diam-mean equicontinuity point of $(\hyper(X),T)$ if and only if it is a $\mathcal{F}$-weak-mean equicontinuity point of $(\hyper(X),T)$.
For this we will use the following notions. 

\begin{definition}
    For a finite family $\mathcal{S}$ of subsets of $X$ we denote 
    $\bigcup \mathcal{S}:=\bigcup_{S\in \mathcal{S}}S$.
    Furthermore, we denote
    $\langle \mathcal{S}\rangle$ for the set of all $A\in \hyper(X)$ with $A\subseteq \bigcup \mathcal{S}$ and $A\cap S\neq \emptyset$ for all $S\in \mathcal{S}$. 
    For $S'\subseteq X$ and a finite family $\mathcal{S}$ of subsets of $X$ we abbreviate $\langle S',\mathcal{S}\rangle:=\langle \{S'\}\cup \mathcal{S}\rangle$. 
\end{definition}

\begin{remark}
    The \emph{Vietoris topology} is the topology on $\hyper(X)$ generated by the sets $\langle\mathcal{U}\rangle$, where $\mathcal{U}$ ranges over all finite families of open subsets of $X$. 
    Since $X$ is a compact metric space the Vietoris topology on $\hyper(X)$ is the topology generated by the Hausdorff metric \cite[Theorem 3.1]{illanes1999hyperspaces}. 
\end{remark}

\begin{remark}
    For a finite family $\mathcal{A}\subseteq \hyper(X)$ the set $\langle\mathcal{A}\rangle$ is a closed subset of $\hyper(X)$. 
\end{remark}



\begin{lemma}
\label{lem:hyper_diameterControlI}
    For $A\in \hyper(X)$, a finite family $\mathcal{B}\subseteq \hyper(X)$ and $C\in \langle A,\mathcal{B}\rangle$ we have 
    \begin{align*}
        d\left(\bigcup \mathcal{B}, C\cap \bigcup \mathcal{B}\right)
        \leq \max_{B\in \mathcal{B}} \diam(B). 
    \end{align*}
\end{lemma}
\begin{proof}
    Denote $\delta:=\max_{B\in \mathcal{B}} \diam(B)$ and 
    $C_\mathcal{B}:=C\cap \bigcup \mathcal{B}$. 
    Clearly, we have $C_\mathcal{B}\subseteq \bigcup \mathcal{B}\subseteq \ball_\delta[\bigcup \mathcal{B}]$. 
    Furthermore, for $x\in \bigcup \mathcal{B}$ there exists $B\in \mathcal{B}$ with $x\in B$. 
    Since $C\in \langle A,\mathcal{B}\rangle$ there exists $y\in B\cap C\subseteq C_\mathcal{B}$. 
    We have
    $d(x,y)\leq \diam(B)\leq \delta$ and hence 
    $x\in \ball_\delta(y)\subseteq \ball_\delta[C_\mathcal{B}]$.
    This shows $\bigcup \mathcal{B}\subseteq \ball_\delta[C_\mathcal{B}]$. 
\end{proof}

\begin{lemma}
    \label{lem:hyper_diameterControlII}
    For $A\in \hyper(X)$ and a finite family $\mathcal{B}\subseteq \hyper(X)$ with $\bigcup\mathcal{B}\subseteq A$ we have 
    \[
    \diam(\langle A, \mathcal{B}\rangle)\leq d\left(A, \bigcup\mathcal{B}\right)+\max_{B\in \mathcal{B}} \diam(B). 
    \]
\end{lemma}
\begin{proof}
    Denote $\delta:=\max_{B\in \mathcal{B}} \diam(B)$.
    For $C\in \langle A,\mathcal{B}\rangle$ we denote 
    $C_\mathcal{B}:=C\cap \bigcup \mathcal{B}$. 
    Consider $C,C'\in \langle A,\mathcal{B}\rangle$. 
    For $x\in C'$ we have $x\in A\cup \bigcup\mathcal{B}=A$. From $C_\mathcal{B}\subseteq C$ and Remark \ref{rem:hausdorff_metric_formulas} we observe 
    \begin{align*}
        d(x,C)
        \leq d(x,C_\mathcal{B})
        \leq \sup_{x'\in A} d(x',C_\mathcal{B})
        \leq d(A,C_\mathcal{B})
        \leq d\left(A,\bigcup \mathcal{B} \right) + d\left(\bigcup \mathcal{B},C_\mathcal{B} \right).
    \end{align*}
    Thus, Lemma \ref{lem:hyper_diameterControlI}  yields
    \begin{align*}
        \sup_{x\in C'}d(x,C)\leq d\left(A,\bigcup \mathcal{B}\right) + \delta.
    \end{align*}
    Similarly, we argue for 
    $
        \sup_{x\in C}d(x,C')\leq d\left(A,\bigcup \mathcal{B}\right) + \delta
    $
    and it follows from Remark~\ref{rem:hausdorff_metric_formulas} that 
    \begin{align*}
        d(C,C')\leq d\left(A,\bigcup \mathcal{B}\right) + \delta. 
    \end{align*}
    Taking the supremum over all $C,C'\in \langle A, \mathcal{B}\rangle$ we obtain the statement. 
\end{proof}

\begin{lemma}
\label{lem:hyper_diameterControlIII}
    For $A\in \hyper(X)$ and a finite family $\mathcal{B}\subseteq \hyper(X)$ with $\bigcup\mathcal{B}\subseteq A$ we have 
    \[
    (T^*\diam)(\langle A, \mathcal{B}\rangle)\leq (T^*d)\left(A, \bigcup\mathcal{B}\right)+\sum_{B\in \mathcal{B}} (T^*\diam))(B). 
    \]
\end{lemma}
\begin{proof}
    Note that  
    $T\langle A, \mathcal{B}\rangle\subseteq \langle TA, T\mathcal{B}\rangle$, where we abbreviate $T\mathcal{B}:=(TB)_{B\in \mathcal{B}}$. 
    Thus, Lemma \ref{lem:hyper_diameterControlII} yields
    \begin{align*}
        (T^*\diam)(\langle A, \mathcal{B}\rangle)
        &=
        \diam(T\langle A, \mathcal{B}\rangle)
        \leq 
        \diam(\langle TA, T\mathcal{B}\rangle)\\
        &\leq d\left(TA,\bigcup_{B\in \mathcal{B}}TB\right)+\max_{B\in \mathcal{B}} \diam(TB)\\
        &\leq (T^*d)\left(A,\bigcup \mathcal{B}\right)+\sum_{B\in \mathcal{B}} (T^*\diam)(B).
    \qedhere
    \end{align*}
\end{proof}

\begin{lemma}
    \label{lem:hyper_meanDiameterControlV}
    Let $(X,T)$ be a dynamical system. 
    For $A\in \hyper(X)$ and a finite family $\mathcal{B}\subseteq \hyper(X)$ with $\bigcup\mathcal{B}\subseteq A$ we have 
    \[
    \Diam_\mathcal{F}(\langle A, \mathcal{B}\rangle)\leq D_\mathcal{F}\left(A, \bigcup\mathcal{B}\right)+\sum_{B\in \mathcal{B}} \Diam_\mathcal{F}(B). 
    \]
\end{lemma}
\begin{proof}
    With a similar argument as above from Lemma \ref{lem:hyper_diameterControlIII} we observe 
    \begin{align*}
        (F_n^*\diam)(\langle A, \mathcal{B}\rangle)
        \leq 
        (F_n^*d)\left(A,\bigcup \mathcal{B}\right)+\sum_{B\in \mathcal{B}} (F_n^*\diam)(B)
    \end{align*}
    for all $n\in \mathbb{N}$. 
    It follows that 
    \begin{align*}
        \Diam_\mathcal{F}(\langle A,\mathcal{B}\rangle)
        &=
        \limsup_{n\to \infty}(F_n^*\diam)(\langle A, \mathcal{B}\rangle)\\
        &\leq 
        \limsup_{n\to \infty}(F_n^*d)\left(A,\bigcup \mathcal{B}\right)+
        \sum_{B\in \mathcal{B}}\limsup_{n\to \infty} (F_n^*\diam)(B)\\
        &=D_\mathcal{F}\left(A,\bigcup \mathcal{B}\right) 
        + \sum_{B\in \mathcal{B}} \Diam_\mathcal{F}(B).
        \qedhere
    \end{align*}
\end{proof}

\begin{proposition}
\label{pro:hyper_pointwiseDMEvsMEvsWME}
    Consider $A\in \hyper(X)$ such that all $x\in A$ are $\mathcal{F}$-diam-mean equicontinuity points of $(X,T)$. The following statements are equivalent. 
    \begin{itemize}
        \item[(i)] $A$ is a $\mathcal{F}$-diam-mean equicontinuity point of $(\hyper(X),T)$.
        \item[(ii)] $A$ is a $\mathcal{F}$-mean equicontinuity point of $(\hyper(X),T)$.
        \item[(iii)] $A$ is a $\mathcal{F}$-weak-mean equicontinuity point of $(\hyper(X),T)$.
    \end{itemize}
\end{proposition}
\begin{proof}
    From Remark \ref{rem:prelims_basicRelationshipWME_ME_DME} we know that (i)$\Rightarrow$(ii)$\Rightarrow$(iii). 
    To show that (iii) implies (i) consider $\epsilon>0$. 
    Since $A$ is a $\mathcal{F}$-weak-mean equicontinuity point of $(\hyper(X),G)$ there exists $\delta>0$ such that for all $C,C'\in \ball_\delta(A)$ we have $\meanorbPM_\mathcal{F}(C,C')<\epsilon/(4\spa(\epsilon/8))$. 
    There exists a finite subset $E\subseteq A$, such that $d(E,A)\leq \delta/2$. 
    Since $E\subseteq A$ consists of $\mathcal{F}$-diam-mean equicontinuity points for any $x\in E$ there exists a compact neighbourhood $B_x$ of $x$ such that 
    $\Diam(B_x)<\epsilon/(2|E|)$. 
    W.l.o.g.\ we assume that $B_x\subseteq \ball_{\delta/2}(x)$. 
    Denote $\mathcal{B}:=(B_x)_{x\in E}$. 
    For $x\in E$ we have $x\in B_x\subseteq \ball_{\delta/2}(x)$ and hence $E\subseteq \bigcup \mathcal{B}\subseteq \ball_{\delta/2}[E]$. 
    It follows that 
    $d(E,\bigcup \mathcal{B})\leq \delta/2$ and hence 
    $d(A,\bigcup\mathcal{B})\leq d(A,E)+d(E,\bigcup\mathcal{B})\leq \delta$. 
    Since clearly $d(A,\ball_\delta[A])\leq \delta$ our choice of $\delta$ yields
    $\meanorbPM_\mathcal{F}\left(\ball_\delta[A],\bigcup\mathcal{B}\right)<\epsilon/(4\spa(\epsilon/8))$. 
    Note that $\bigcup\mathcal{B}\subseteq \ball_\delta[A]$. 
    It follows from Lemma \ref{lem:hyperWeylVsMeanOrb} that
    \begin{align*}
        D_\mathcal{F}\left(\ball_\delta[A],\bigcup\mathcal{B}\right)\leq \spa\left(\tfrac{\epsilon}{8}\right)\mathcal{W}_\mathcal{F}\left(\ball_\delta[A],\bigcup\mathcal{B}\right)+2\frac{\epsilon}{8}
        <\frac{\epsilon}{4} + \frac{\epsilon}{4}
        =\frac{\epsilon}{2}. 
    \end{align*}
    From Lemma \ref{lem:hyper_meanDiameterControlV} we conclude
    \begin{align*}
        \Diam_\mathcal{F}\left(\langle\ball_\delta[A],\mathcal{B}\rangle\right)
        &\leq D_\mathcal{F}\left(\ball_\delta[A],\bigcup\mathcal{B}\right)+\sum_{x\in E}\Diam_\mathcal{F}(B_x)\\
        &< \frac{\epsilon}{2} + \sum_{x\in E}\frac{\epsilon}{2|E|}= \epsilon.
    \end{align*}
    From $A\subseteq \ball_\delta[A]^\circ$ and $A\cap B^\circ\neq \emptyset$ for all $B\in \mathcal{B}$ we observe that 
    \begin{align*}
        A\in \langle \ball_\delta[A]^\circ,(B^\circ)_{B\in \mathcal{B}}\rangle\subseteq \langle\ball_\delta[A],\mathcal{B}\rangle.
    \end{align*}
    Thus, $\langle\ball_\delta[A],\mathcal{B}\rangle$ is a compact neighbourhood of $A$ with 
    \begin{align*}
    \Diam_\mathcal{F}\left(\langle\ball_\delta[A],\mathcal{B}\rangle\right)
    &<\epsilon.    
    \qedhere
    \end{align*}
\end{proof}

\begin{proof}[Proof of Theorem \ref{the:intro_Main}(\ref{intro_Main_H(X)}):]
    By Remark \ref{rem:prelims_basicRelationshipWME_ME_DME} it remains to show that $(\hyper(X),T)$ is diam-mean equicontinuous whenever $(\hyper(X),T)$ is weakly-mean equicontinuous. 
    If $(\hyper(X),T)$ is weakly-mean equicontinuous, then for any $x\in X$ we have that $\{x\}$ is a $\mathcal{F}$-weak mean equicontinuity point of $(\hyper(X),T)$ and Proposition \ref{pro:hyper_HXwmeImpliesXdme} yields that $x$ is a $\mathcal{F}$-diam-mean equicontinuity point of $(X,T)$.
    Thus, for any $A\in \hyper(X)$ we have that $A$ consists of $\mathcal{F}$-diam-mean equicontinuity points (w.r.t.\ $(X,T)$) and that $A$ is a $\mathcal{F}$-weak-mean equicontinuity point of $(\hyper(X),T)$. 
    From Proposition~\ref{pro:hyper_pointwiseDMEvsMEvsWME} it follows that $A$ is a $\mathcal{F}$-diam-mean equicontinuity point of $(\hyper(X),T)$. 
\end{proof}

\section{Almost diam-mean equicontinuity} \label{sec:almostDME}

A dynamical system $(X,T)$ is called \emph{almost $\mathcal{F}-$diam-mean equicontinuous} if and only if the set of $\mathcal{F}$-diam-mean equicontinuity points is dense in $X$. For further details on this notion see \cite{garcia2021mean}, where it is called {almost diam-mean equicontinuity}. 

Note that whenever $D\subseteq X$ is dense, then $\{\delta_x;\, x\in D\}$ is dense in $\{\delta_x;\, x\in X\}$ and it follows from a standard application of the Krein-Milman theorem that the set of all finite convex combinations of elements of $\{\delta_x;\, x\in D\}$ is dense in $\myper(X)$. 
Thus, the set of all $\mu\in \myper(X)$ with $\supp(\mu)\subseteq D$ is dense in $\myper(X)$. 
Combining this observation with Proposition \ref{pro:supportAndDiamMeanEquicontinuousPoints} we obtain the following. 

\begin{corollary}
\label{cor:MXalmostDME}
    If $(X,T)$ is an almost $\mathcal{F}$-diam-mean equicontinuous dynamical system, then $(\myper(X),T)$ is almost $\mathcal{F}$-diam-mean equicontinuous. 
\end{corollary}

In consideration of our previous discussion it is unexpected that the analog holds for $(\hyper(X),T)$, which we present next. 

\begin{lemma}
\label{lem:almostDMEHX_I}
    For finite $E\in \hyper(X)$, $\delta>0$ and $A\in \ball_\delta(E)$ we have 
    \[
        d(TA,TE)\leq 2\max_{x\in E}\diam(T\ball_\delta(x)).
    \]
\end{lemma}
\begin{proof}
    Denote $\epsilon:=\max_{x\in E}\diam(T\ball_\delta(x))$. 
    For $x\in E$ we have $Tx\in T\ball_\delta(x)$ and 
    $\diam(T\ball_\delta(x))\leq \epsilon$ yields $T\ball_\delta(x)\subseteq \ball_\epsilon(Tx)$. 
    From $A\in \ball_\delta(E)$ we observe $A\subseteq \ball_\delta[E]$ and $E\subseteq \ball_\delta[A]$. 
    Thus, 
    \begin{align*}
        TA
        \subseteq T\ball_\delta[E]
        =T\left(\bigcup_{x\in E}\ball_\delta(x)\right)
        =\bigcup_{x\in E}T\ball_\delta(x)
        \subseteq \bigcup_{x\in E}\ball_\epsilon(Tx)
        =\ball_\epsilon[TE].
    \end{align*}
    Furthermore, from $E\subseteq \ball_\delta[A]$ we know that for $x\in E$ there exists $a_x\in A$ with $d(a_x,x)\leq \delta$, i.e.\ $a_x\in \ball_\delta(x)$. 
    It follows that $Ta_x\in T\ball_\delta(x)\subseteq \ball_\epsilon(Tx)$. i.e.\ $Tx\in \ball_\epsilon(Ta_x)$.
    Thus, 
    \[TE=\bigcup_{x\in E}\{Tx\}\subseteq \bigcup_{x\in E}\ball_\epsilon(Ta_x)\subseteq \ball_\epsilon[TA].\]
    This shows $d(TA,TE)\leq \epsilon$. 
\end{proof}

\begin{lemma}
\label{lem:almostDMEHX_II}
    For finite $E\in \hyper(X)$ and $\delta>0$ we have 
    \[(T^*\diam)(\ball_\delta(E))\leq 2\max_{x\in E}(T^*\diam)(\ball_\delta(x)).\]
\end{lemma}
\begin{proof}
    Denote 
    \[
        \epsilon:=\max_{x\in E}(T^*\diam)(\ball_\delta(x))=\max_{x\in E}\diam(T\ball_\delta(x)).
    \]
    For $C,C'\in T\ball_\delta(E)$ choose $A,A'\in \ball_\delta(E)$ with $TA=C$ and $TA'=C'$. From Lemma \ref{lem:almostDMEHX_I} we observe 
    \begin{align*}
        d(C,C') = d(TA,TA')
        \leq d(TA,TE) + d(TE,TA')
        \leq 2\epsilon. 
    \end{align*}
    It follows that $T^*\diam(\ball_\delta(E))=\diam(T\ball_\delta(E))\leq 2\epsilon$. 
\end{proof}

\begin{lemma}
\label{lem:almostDMEHX_III}
    For finite $E\in \hyper(X)$ and $\delta>0$ we have 
    \[\Diam_\mathcal{F}(\ball_\delta(E))\leq 2\sum_{x\in E}\Diam_\mathcal{F}(\ball_\delta(x)).\]
\end{lemma}
\begin{proof}
    From Lemma \ref{lem:almostDMEHX_II} applied to $(X^k,T)$ we observe that
    \[((T^k)^*\diam)(\ball_\delta(E))\leq 2\max_{x\in E}((T^k)^*\diam)(\ball_\delta(x))
    \leq 2\sum_{x\in E}((T^k)^*\diam)(\ball_\delta(x))\]
     for all $k\in \mathbb{N}_0$ and hence 
    \[(F_n^*\diam)(\ball_\delta(E))\leq 2\sum_{x\in E}(F_n^*\diam)(\ball_\delta(x))\]
    for all $n\in \mathbb{N}$.
    It follows that 
    \begin{align*}
        \Diam_\mathcal{F}(\ball_\delta(E))
        &=\limsup_{n\to \infty}
        (F_n^*\diam)(\ball_\delta(E))\\
        &\leq \limsup_{n\to \infty}2\sum_{x\in E}(F_n^*\diam)(\ball_\delta(x))\\
        &\leq 2\sum_{x\in E}\limsup_{n\to \infty}(F_n^*\diam)(\ball_\delta(x))\\
        &=2\sum_{x\in E}(F_n^*\diam)(\ball_\delta(x)). 
        \qedhere
    \end{align*}
\end{proof}

\begin{proposition}
    Let $(X,T)$ be a dynamical system and consider a finite $E\in \hyper(X)$. If $E$ consists of $\mathcal{F}$-diam-mean equicontinuous points, then $E$ is a $\mathcal{F}$-diam-mean equicontinuous point of $(\hyper(X),T)$. 
\end{proposition}
\begin{proof}
    Let $\epsilon>0$ and choose $\delta>0$ such that for all $x\in E$ we have $\Diam_\mathcal{F}(\ball_\delta(x))\leq \epsilon/(2|E|)$. 
    From Lemma \ref{lem:almostDMEHX_III} we observe 
    \begin{align*}
        \Diam_\mathcal{F}(\ball_\delta(E))
        &\leq 2\sum_{x\in E}\Diam_\mathcal{F}(\ball_\delta(x))
        <\epsilon.   
        \qedhere
    \end{align*}
\end{proof}

Note that whenever $D\subseteq X$ is dense, then the set of all finite subsets of $D$ is dense in $\hyper(X)$. We conclude the following. 

\begin{corollary}
\label{cor:HXalmostDME}
    If $(X,T)$ is an almost $\mathcal{F}$-diam-mean equicontinuous dynamical system, then also $(\hyper(X),T)$ is almost $\mathcal{F}$-diam-mean equicontinuous. 
\end{corollary}

\begin{remark}
    From our previous discussion it follows that for a Denjoy system $(\mathbb{T},T)$ we have that $(\hyper(\mathbb{T}),T)$ is an almost $\mathcal{F}$-diam-mean equicontinuous system with infinite entropy. 
\end{remark}

\begin{remark}
    It remains an open question, whether the converse holds in the Corollaries \ref{cor:MXalmostDME} and \ref{cor:HXalmostDME}. 
\end{remark}

\section{Actions of locally compact and $\sigma$-compact amenable groups}\label{sec:groups}

The notions of weak-mean equicontinuity, mean equicontinuity and diam-mean equicontinuity have been extensively studied in the context of group actions \cite[Section~8]{li2021meanEquicontinuity}. See also 
\cite{fuhrmann2022structure, xu2024weak, fuhrmann2025continuity, hauser2025meandiameterregularitydiammean} and the references within. We next provide some background and present that the results presented so far can also be obtained in this setting. 

Let $G$ be a locally compact and $\sigma$-compact group and choose a Haar measure $\haar{\cdot}=\int (\cdot)\dd g$ on $G$. 
A sequence $(F_n)_{n\in \mathbb{N}}$ of compact subsets of $G$ with $\haar{F_n}>0$ is called \emph{(left) F{\o}lner} if for all non-empty compact subsets $K\subseteq G$ we have 
\[\lim_{n\to \infty} \haar{KF_n\Delta F_n}/\haar{F_n}\to 0.\]
$G$ is called \emph{amenable} if it allows for a F{\o}lner sequence.


Let $X$ be a compact metric space. An \emph{action} of $G$ on $X$ is a group homomorphism $\alpha$ from $G$ into the group of homeomorphisms on $X$ such that the induced map $G\times X\to X, ~(g,x)\mapsto gx:=\alpha(g)(x)$ is continuous.
In the following we denote $(X,G)$ for an action and keep the group homomorphism implicit. 
We recommend \cite{auslander1988minimal} for a detailed exposition on actions. 



Let $(X,G)$ be an action.
Identifying $g$ with the homeomorphism $X\to X$ allows to write $g(x,x'):=(gx,gx')$,  $g^*f:=f\circ g$, $g\mu:=g_*\mu$, $gA:=g(A)$ for $(x,x')\in X^2$, $f\in C(X)$, $\mu\in \myper(X)$ and $A\in \hyper(X)$. 
The respective maps $G\times X^2\to X^2$, $G\times \myper(X)\to \myper(X)$ and $G\times \hyper(X)\to \hyper(X)$ are also actions, for which we denote $(\myper(X),G)$ and $(\hyper(X),G)$. 
We call $A\in \hyper(X)$ \emph{invariant} if $gA=A$ holds for all $g\in G$. 
Note that an invariant $A\in \hyper(X)$ allows for a restriction of the action of $G$. The respective action $(A,G)$ is called a \emph{subaction of $(X,G)$}. 

\subsection{$\mathcal{F}$-(diam-/weak-)mean equicontinuity}
Consider $f\in C(X)$ and $\mu\in \myper(X)$. 
For $F\subseteq G$ compact with $\haar{F}>0$ we denote 
\[F^*f:=\frac{1}{\haar{F}}\int_F (g^*f)\dd g
\hspace{1cm}\text{and}\hspace{1cm}
F_*\mu:=\frac{1}{\haar{F}}\int_F (g.\mu) \dd g\]
for the respective \emph{Cesàro averages}. 
We have  
$F^*f\in C(X)$, 
$F_*\mu\in \myper(X)$ and 
$(F_*\mu)(f)=\mu(F^*f)$. 

For further reading on the following notions see Subsection \ref{subsec:diamWeakMeanEquicontinuity}, \cite{fuhrmann2022structure, 
fuhrmann2025continuity} 
and the references within. 
Let $\mathcal{F}=(F_n)_{n\in \mathbb{N}}$ be a F\o lner sequence in $G$. The \emph{$\mathcal{F}$-mean-orbital pseudometric} is defined by  
\begin{align*}
    \meanorbPM_\mathcal{F}(x,y):=\limsup_{n\to \infty}d((F_n)_*\delta_x,(F_n)_*\delta_y). 
\end{align*}
Note that $d\in C(X^2)$ and consider the action $(X^2,G)$. 
The \emph{$\mathcal{F}$-mean pseudo-metric} is defined by 
\begin{align*}
    \meanPM_\mathcal{F}(x,y):=\limsup_{n\to \infty}(F_n^*d)(x,y). 
\end{align*}
Note that $\diam\in C(\hyper(X))$ and consider the action $(\hyper(X),G)$. 
For $A\in \hyper(X)$ the \emph{$\mathcal{F}$-mean diameter} is defined by 
\begin{align*}
    \Diam_\mathcal{F}(A):=\limsup_{n\to \infty}(F_n^*\diam)(A). 
\end{align*}

\begin{lemma}
\label{lem:meanPMDiam_relations}
    For $x,y\in X$ we have 
    $
        \meanorbPM_\mathcal{F}(x,y)
        \leq \meanPM_\mathcal{F}(x,y)
        =\Diam_\mathcal{F}(\{x,y\})
    $.
\end{lemma}
\begin{proof}
    It is straightforward to observe 
    $\meanPM_\mathcal{F}(x,y)=\Diam_\mathcal{F}(\{x,y\})$. 
    For $n\in \mathbb{N}$ we have $(F_n)_*\delta_{(x,y)}\in \Pi((F_n)_*\delta_x,(F_n)_*\delta_y)$ and the definition of the Wasserstein metric yields
    $
        d((F_n)_*\delta_{x},(F_n)_*\delta_{y})
        \leq (F_n)_*\delta_{(x,y)}(d)
        =F_n^*d(x,y).
    $
    Taking the limit superior along $n$ we observe 
    $\meanorbPM_\mathcal{F}(x,y) \leq \meanPM_\mathcal{F}(x,y)$. 
\end{proof}

Given a F\o lner sequence $\mathcal{F}$ in $G$ a point $x\in X$ is said to be a
\begin{itemize}
    \item 
    \emph{$\mathcal{F}$-weak-mean equicontinuity point} if for all $\epsilon>0$ there exists $\delta>0$ such that for all $y\in \ball_\delta(x)$ we have $\meanorbPM_\mathcal{F}(x,y)<\epsilon$. 
    \item 
    \emph{$\mathcal{F}$-mean equicontinuity point} if for all $\epsilon>0$ there exists $\delta>0$ such that for all $y\in \ball_\delta(x)$ we have $\meanPM_\mathcal{F}(x,y)<\epsilon$.
    \item 
    \emph{$\mathcal{F}$-diam-mean equicontinuity point} if for all $\epsilon>0$ there exists $\delta>0$ such that $\Diam_\mathcal{F}(\ball_\delta(x))<\epsilon$. 
\end{itemize}
$(X,G)$ is called $\mathcal{F}$-weakly-mean equicontinuous if all $x\in X$ are $\mathcal{F}$-weak-mean equicontinuity points. Similarly, we define the notion of $\mathcal{F}$-mean equcontinuity and $\mathcal{F}$-diam-mean equicontinuity for an action.
An action $(X,G)$ is called \emph{almost $\mathcal{F}$-diam-mean equicontinuous} if the set of all $\mathcal{F}$-diam-mean equicontinuity points is dense in $X$. 
Using the Fubini theorem it is straightforward to adapt the respective proofs above to obtain the following. 

\begin{theorem}
\label{the:mainTheoremForActionsFolnerDep}
    Let $(X,G)$ be an action and $\mathcal{F}$ a F\o lner sequence in $G$. 
    \begin{enumerate}
    
        \item 
        The following statements are equivalent.
        \begin{itemize}
            \item[(i)] $(X,G)$ is $\mathcal{F}$-mean equicontinuous.
            \item[(ii)] $(\myper(X),G)$ is $\mathcal{F}$-mean equicontinuous. 
            \item[(iii)] $(\myper(X),G)$ is $\mathcal{F}$-weakly-mean equicontinuous. 
        \end{itemize}
        
        \item 
        The following statements are equivalent.
        \begin{itemize}
            \item[(i)] $(X,G)$ is $\mathcal{F}$-diam-mean equicontinuous.
            \item[(ii)] $(\myper(X),G)$ is $\mathcal{F}$-diam-mean equicontinuous. 
        \end{itemize}
        
        \item 
        The following statements are equivalent.
        \begin{itemize}
            \item[(i)] $(\hyper(X),G)$ is $\mathcal{F}$-diam-mean equicontinuous.
            \item[(ii)] $(\hyper(X),G)$ is $\mathcal{F}$-mean equicontinuous. 
            \item[(iii)] $(\hyper(X),G)$ is $\mathcal{F}$-weakly-mean equicontinuous. 
        \end{itemize}
        \item 
        If $(\hyper(X),G)$ is $\mathcal{F}$-diam-mean equicontinuous, then $(X,G)$ is $\mathcal{F}$-diam-mean equicontinuous. 

        \item 
        If $(X,G)$ is almost $\mathcal{F}$-diam-mean equicontinuous, then also $(\myper(X),G)$ and $(\hyper(X),G)$ are almost $\mathcal{F}$-diam-mean equicontinuous. 
    \end{enumerate}
\end{theorem}

\subsection{(Diam-/weak-)mean equicontinuity}
Taking the supremum over all \linebreak 
F\o lner sequences $\mathcal{F}$ we define
the \emph{mean diameter} 
$\Diam:=\sup_{\mathcal{F}}\Diam_\mathcal{F}$, 
the \emph{mean pseudometric}
$D:=\sup_\mathcal{F}D_\mathcal{F}$, 
and the \emph{mean-orbital pseudometric}
$\meanorbPM:=\sup_\mathcal{F}\meanorbPM_\mathcal{F}$. 
As above we define the notions of an diam-mean equicontinuous point, a mean equicontinuous point and a weak-mean equicontinuous point via $\meanorbPM$, $\meanPM$, $\Diam$, as well as the respective notions for actions. We call an action \emph{almost diam-mean equicontinuous} if it allows for a dense set of diam-mean equicontinuity points. 

\begin{remark}
    In \cite{fuhrmann2025continuity, hauser2026maximal} it is shown that in general the notions of $\mathcal{F}$-mean equicontinuity and $\mathcal{F}$-weak-mean equicontinuity depend on the choice of a F\o lner sequence $\mathcal{F}$.
\end{remark}

\begin{remark}
    Whenever $G$ is Abelian or $(X,G)$ minimal, then $(X,G)$ is mean equicontinuous as soon as it is $\mathcal{F}$-mean equicontinuous for some F\o lner sequence $\mathcal{F}$ \cite[Section 5]{fuhrmann2022structure}. 
    Thus, for a homeomorphism $T\colon X\to X$ we observe that $(X,T)$ is mean equicontinuous (as defined in the introduction) if and only if the respective action $(X,\mathbb{Z})$ given by $kx:=T^k(x)$ is mean equicontinuous (as defined in this section). 
\end{remark}

\begin{remark}
    A similar statement holds for the concept of weak-mean equicontinuity \cite[Section 4]{xu2024weak}.
\end{remark}

\begin{remark}
    For \textbf{minimal} actions a similar statement holds for the concept of diam-mean equicontinuity \cite{garcia2021mean, haupt2025multivariate}. 
    However, it needs to be highlighted that for non-minimal homeomorphisms $T\colon X\to X$ a potential conflict of notation arises, since it is an open question whether the diam-mean equicontinuity of $(X,T)$ as defined in the introduction is equivalent to the diam-mean equicontinuity of the respective action $(X,\mathbb{Z})$ as defined in this section. 
    Note that the latter is also called \emph{Banach diam-mean equicontinuity} in the context of dynamical systems $(X,T)$ \cite{garcia2021mean}. 
\end{remark}

Recall from \cite[Proposition 3.3]{fuhrmann2025continuity} that an action is weakly mean equicontinuous if and only if it is $\mathcal{F}$-weakly mean equicontinuous for any F\o lner sequence $\mathcal{F}$.
A similar statement holds for the notions of mean-equicontinuity \cite[Lemma 2.10]{hauser2026maximal} and diam-mean equicontinuity \cite[Proposition 6.9]{hauser2025meandiameterregularitydiammean}. 
The proofs of the following point-wise statements are minor modifications of the respective proofs from the literature. We include the short argument for the convenience of the reader. 

\begin{lemma}
\label{lem:pointwiseFolnerForAll}
    Let $(X,G)$ be an action and $x\in X$. 
    \begin{itemize}
        \item[(i)] $x$ is a diam-mean equicontinuity point if and only if $x$ is a $\mathcal{F}$-diam-mean equicontinuity point for any F\o lner sequence $\mathcal{F}$ in $G$. 
        \item[(ii)] $x$ is a mean equicontinuity point if and only if $x$ is a $\mathcal{F}$-mean equicontinuity point for any F\o lner sequence $\mathcal{F}$ in $G$. 
        \item[(iii)] $x$ is a weak-mean equicontinuity point if and only if $x$ is a $\mathcal{F}$-weak-mean equicontinuity point for any F\o lner sequence $\mathcal{F}$ in $G$.
    \end{itemize}
\end{lemma}
\begin{proof}
(i): 
    See \cite[Proposition 6.9]{hauser2025meandiameterregularitydiammean}, where the full proof of the point-wise statement is presented. 

(ii):
    If $x$ is not a mean equicontinuity point, then there exist $\delta>0$ and 
    a sequence $(x_n)_{n\in \mathbb{N}}$ in $X$ with $x_n\to x$ and $D(x_n,x)\geq 2\delta$.
    For $n\in \mathbb{N}$ choose a F\o lner sequence $\mathcal{F}^{(n)}$ with 
    $D_{\mathcal{F}^{(n)}}(x_n,x)\geq \delta$. 
    W.l.o.g.\ we assume that $D_{\mathcal{F}^{(n)}}(x_n,x)=\lim_{k\to \infty}((F^{(n)}_k)^*d)(x_n,x)$. 
    By \cite[Proposition 2.1]{fuhrmann2025continuity} there exists a F\o lner sequence $\mathcal{F}$ that shares a common subsequence with all $\mathcal{F}^{(n)}$. 
    In particular, we have $D_\mathcal{F}(x_n,x)\geq D_{\mathcal{F}^{(n)}}(x_n,x)\geq \delta$. 
    Thus, $x$ is not a $\mathcal{F}$-mean equicontinuity point.

(iii): 
    Follows from a similar argument as (ii). 
\end{proof}

In particular, it follows that an action is almost diam-mean equicontinuous iff it is is almost $\mathcal{F}$-diam-mean equicontinuous w.r.t.\ any F\o lner sequence $\mathcal{F}$. 
Thus, combining Lemma \ref{lem:pointwiseFolnerForAll} with Theorem \ref{the:mainTheoremForActionsFolnerDep} we observe the following. 

\begin{corollary}
    Let $G$ be a locally compact $\sigma$-compact amenable group. 
    For an action $(X,G)$ the statements of Theorem \ref{the:mainTheoremForActionsFolnerDep} also hold w.r.t.\ the notion of (almost) diam-mean equicontinuity (for actions), mean equicontinuity and weak-mean equicontinuity. 
\end{corollary}

\appendix
\refstepcounter{section}
\section*{Appendix: Weak mixing of $\myper(X)$}

A dynamical system $(X,T)$ is called \emph{(topological) transitive} if for all non-empty and open $U,V\subseteq X$ there exists $k\in \mathbb{N}$ such that $T^kU\cap V\neq \emptyset$. 
It is called \emph{weakly (topologically) mixing} if $(X^2,T)$ is transitive, i.e.\ if for open and non-empty $U_1,U_2,V_1,V_2\subseteq X$ there exists $k\in \mathbb{N}$ with $T^k V_i\cap U_i\neq \emptyset$ for $i=1,2$. Note that any weakly mixing dynamical system is transitive. 
A well known result of Furstenberg \cite[Proposition II.3]{furstenber1967disjointness} shows that if $(X,T)$ is weakly mixing, then for all $n\in \mathbb{N}$ also $(X^n,T)$ is weakly mixing.

It was shown in 
\cite{bauer1975topological} 
and
\cite{banks2005chaos} 
that $(X,T)$ is weakly mixing, iff $(\hyper(X),T)$ is transitive. 
Furthermore, it was shown in 
\cite{bauer1975topological}
that $(X,T)$ is weakly mixing, iff $(\myper(X),T)$ is weakly mixing. 
We next show that on $(\myper(X),T)$ weak mixing and transitivity are equivalent. For this we need the following lemma. 

\begin{lemma}\cite[Lemma 5]{banks1999topological}
\label{lem:weakMixingAlaBanks}
    A dynamical system $(X,T)$ is weakly mixing if and only if for all open and non-empty $U,V_1,V_2\subseteq G$ there exists $k\in \mathbb{N}$ with $T^kU\cap V_i\neq \emptyset$ for $i=1,2$.  
\end{lemma}

\begin{proposition}
\label{pro:weakMixingForMX}
    For a dynamical system $(X,T)$ the following statements are equivalent. 
    \begin{itemize}
        \item[(i)] $(X,T)$ is weakly mixing. 
        \item[(ii)] $(\mathcal{M}(X),T)$ is weakly mixing. 
        \item[(iii)] $(\mathcal{M}(X),T)$ is transitive. 
    \end{itemize}
\end{proposition}
\begin{remark}
    The equivalence of (i) and (ii) was shown in \cite{bauer1975topological}. 
    We include a short and slightly different proof for (i)$\Rightarrow$(ii) for the convenience of the reader. 
\end{remark}
\begin{proof}
(i)$\Rightarrow$(ii):    
    For $i\in \mathbb{N}$ denote $X_i:=X^{2^i}$ 
    and consider
    $\phi_i\colon X_i\to \mathcal{M}(X)$
    given by $(x_1,\dots, x_{2^i})\mapsto \frac{1}{2^i} \sum_{j=1}^{2^i} \delta_{x_j}$. 
    Clearly, $\phi_i$ is continuous and equivariant and hence a factor map onto its image $M_i:=\phi_k(X_i)$. 
    Note that by our assumption $(X^{2^i},T)$ is weakly mixing and hence also its factor $(M_i,T)$ is weakly mixing and that $M_i\subseteq M_{i+1}$ for $i\in \mathbb{N}$. 
    Since $M:=\bigcup_{i\in \mathbb{N}}M_i$ is dense in the convex hull of $\{\delta_x;\, x\in X\}$ it follows from a well known application of the Krein-Milman theorem that 
    $M$ is dense in $\mathcal{M}(X)$.

    Consider $n\in \mathbb{N}$. 
    For open and non-empty subsets $\mathfrak{U}_1,\mathfrak{U}_2,\mathfrak{V}_1,\mathfrak{V}_2\subseteq \mathcal{M}(X)$ we find $i\in \mathbb{N}$ with $U_j:=M_i\cap \mathfrak{U}_j\neq \emptyset$ and $V_j:=M_i\cap \mathfrak{V}_j\neq \emptyset$ for $j=1,2$. 
    Since $(M_i,T)$ is weakly mixing there exists $k\in \mathbb{N}$ with $\emptyset\neq T^k U_j\cap V_j\subseteq T^k\mathcal{U}_j\cap \mathcal{V}_j$ for $j=1,2$. 
    This shows $(\mathcal{M}(X),T)$ to be weakly mixing. 

(ii)$\Rightarrow$(iii):
    Trivial. 

(iii)$\Rightarrow$(i): 
    To show that $(X,T)$ is weakly mixing we will use Lemma \ref{lem:weakMixingAlaBanks} and consider non-empty open sets $U,V_1,V_2\subseteq X$. 
    Denote $\mathfrak{U}$ for the set of all $\mu\in \mathcal{M}(X)$ with $\mu(U)>3/4$.
    Denote $\mathfrak{V}$ for the set of all $\mu\in \mathcal{M}(X)$ with $\mu(V_i)>1/3$ for $i=1,2$. 
    It is a straight forward consequence of the Portmanteau Theorem that 
    $\mathfrak{U}$ and $\mathfrak{V}$ are open (non-empty) subsets of $\myper(X)$. 
    Since $(\myper(X),T)$ is assumed to be transitive we find $k\in \mathbb{N}$ with 
    $T^k\mathfrak{U}\cap \mathfrak{V}\neq \emptyset$. 
    Thus, there exists $\mu\in \mathfrak{U}$ with $T^k\mu\in \mathfrak{V}$. 
    For $i=1,2$ we observe 
    $\mu(U)>3/4$ and $\mu((T^k)^{-1}(V_i))=T^k\mu(V_i)>1/3$.
    Thus, there exists $x\in U\cap (T^k)^{-1}(V_i)$
    and  
    $T^kx\in T^kU \cap V_i$ yields $T^kU \cap V_i\neq \emptyset$. 
    This shows $(X,T)$ to be weakly mixing. 
\end{proof}

\begin{remark}
    The proofs of \cite[Proposition~II.3]{furstenber1967disjointness} and \cite[Lemma~5]{banks1999topological} allow for a straight forward generalization to actions of Abelian groups.
    With the proofs presented above and in \cite{banks2005chaos} it follows that Proposition \ref{pro:weakMixingForMX} as well as the analog statement for $\hyper(X)$ also hold for actions of Abelian groups. 
    For non-Abelian groups the situation is significantly more complicated. See \cite[Page~169]{peleg1972weak} for a weakly mixing action $(X,G)$ of a non-Abelian group $G$ for which $(X^3,G)$ is not transitive. 
    
    An action $(X,G)$ is called \emph{weakly mixing of all orders} if $(X^n,G)$ is transitive for all $n\in \mathbb{N}$. 
    It is not difficult to adapt the proofs above and in \cite{bauer1975topological} to see that for any group $G$ an action $(X,G)$ is weakly mixing of all orders iff $(\myper(X),G)$ is weakly mixing of all orders, iff $(\hyper(X),G)$ is weakly mixing of all orders.
\end{remark}

\bibliographystyle{alpha}
\bibliography{ref}
\end{document}